\documentclass{article}

\usepackage{xcolor}
\usepackage[affil-it]{authblk}
\usepackage{enumitem}
\usepackage{setspace}
\usepackage{longtable}
\newcommand{\s}{1.5} %LINE SPACING ENUMERATE QUOTE ENVIRONMENT

\usepackage{amsthm}

\usepackage{amssymb, amsmath,multicol}
\usepackage{synttree, bussproofs}
\newtheorem{lemma}{Lemma}[]
\newtheorem{corollary}{Corollary}[]
\newtheorem{definition}{Definition}[]
\newtheorem{theorem}{Theorem}[]
\renewcommand{\phi}{\varphi}
\newcommand{\dfn}{Definition}
\newcommand{\sect}{Section}
\newcommand{\fig}{Figure}

\newcommand{\lem}{Lemma}

\newcommand{\hop}{\mathsf{H}}

\newcommand{\quasiframe}{$\langlae$-frame}
\newcommand{\quasimodel}{$\langlae$-model}
\newcommand{\laeframe}{$\lae$-frame}
\newcommand{\laemodel}{$\lae$-model}
\newcommand{\quasiframes}{$\langlae$-frames}

\newcommand{\ifandonlyif}{\textit{iff} }

\newcommand{\laei}{\mathsf{LAE}}
\renewcommand{\infty}{\ast}

\newcommand{\lae}{\mathsf{TLAE}}%^{*}}}
\newcommand{\act}[1]{\mathfrak{d}^{\alpha_{#1}}_{j}}
\newcommand{\actag}[1]{\mathfrak{d}^{\alpha_{i}}_{#1}}

\newcommand{\expec}[1]{\mathfrak{e}^{\alpha_{#1}}}

\newcommand{\ques}{\langle ? \rangle}

\newcommand{\bldia}{\bdia}
\newcommand{\hdia}{\mathsf{P}}

\newcommand{\hboxx}{\mathsf{H}}

\newcommand{\w}[1]{W_{#1}}

\usepackage{bm}

\newcommand{\lang}{\langlae}%\mathcal{L}}

\newcommand{\wdia}{\Diamond}
\newcommand{\ndia}{ \raisebox{0.13ex}{\scalebox{0.83}{$\langle$\hspace{-0.45pt}\scalebox{0.9}{\normalfont\textsf{A}}\hspace{-0.45pt}$\rangle$}} }
\newcommand{\bdia}{\blacklozenge}
\newcommand{\qdia}{\langle ? \rangle}
\newcommand{\h}{\mathsf{H}}
\newcommand{\wbox}{\Box}
\newcommand{\nbox}{ \raisebox{0.3ex}{\scalebox{0.74}{$\bm{[}$\scalebox{1}{\hspace{0.6pt}\raisebox{-0.18ex}{\normalfont\textsf{A}}\hspace{0.6pt}}$\bm{]}$}} }
\newcommand{\nboxs}{\hspace{-3pt}\scalebox{0.8}{\nbox}\hspace{-3pt}}
\newcommand{\bbox}{\blacksquare}
\newcommand{\qbox}{[?]}
\newcommand{\p}{\mathsf{P}}
\newcommand{\atomacts}{\mathtt{Action}}
\newcommand{\actions}{\mathtt{Action}^{*}}
\renewcommand{\actions}{\langact}
\newcommand{\agents}{\mathtt{Agent}}
\newcommand{\var}{\mathtt{Var}}

\newcommand{\cl}[1]{Cl({#1})}
\newcommand{\at}[1]{At({#1})}
\newcommand{\wit}[1]{\mathtt{Wit}^{#1}}
\newcommand{\witset}{\mathtt{Wit}}
\definecolor{darkolivegreen}{rgb}{0.33, 0.42, 0.18}
\definecolor{dollarbill}{rgb}{0.52, 0.73, 0.4}

\newcommand{\langact}{\mathcal{L}_{\mathsf{Act}}}
\newcommand{\langlae}{\mathcal{L}_{\lae}}

\usepackage{tikz}
\usetikzlibrary{trees}
\usetikzlibrary{fit}
\usepackage{tikz}
\usetikzlibrary{positioning,arrows,calc,decorations.markings}

\newcommand{\instr}[3]{\textit{Instr}^{#1}_{#2}({#3})}
\newcommand{\achiev}[4]{\textit{Achiev}^{#1}_{#2}({#3},{#4})}
\newcommand{\fail}[4]{\textit{Fail}^{#1}_{#2}({#3},{#4})}
\newcommand{\success}[4]{\textit{Succ}^{#1}_{#2}({#3},{#4})}
\newcommand{\best}[4]{\textit{Best}^{#1}_{#2}({#3},{#4})}
\newcommand{\worst}[4]{\textit{Worst}^{#1}_{#2}({#3},{#4})}
\newcommand{\better}[2]{\succeq^{#1}_{#2}}
\newcommand{\sbetter}[2]{\succ^{#1}_{#2}}
\newcommand{\Rbb}{R_{\bbox}}
\newcommand{\Rwb}{R_{\wbox}}
\newcommand{\Rh}{R_{\h}}
\newcommand{\good}[3]{\textit{Good}^{#1}_{#2}({#3})}
\newcommand{\poor}[3]{\textit{Poor}^{#1}_{#2}({#3})}

\newcommand{\matnew}[1]{\textcolor{black}{#1}}

\newcounter{quotecount}
\renewcommand{\theequation}{$\mathsf{d\arabic{quotecount}}$}

\tikzset{
modal/.style={>=stealth',shorten >=1pt,shorten <=1pt,auto,node distance=1.5cm,semithick},
world/.style={circle,draw,minimum size=0.5cm,fill=gray!15},
point/.style={circle,draw,inner sep=0.5mm,fill=black},
reflexive above/.style={->,loop,looseness=7,in=120,out=60},
reflexive below/.style={->,loop,looseness=7,in=240,out=300},
reflexive left/.style={->,loop,looseness=7,in=150,out=210},
reflexive right/.style={->,loop,looseness=7,in=30,out=330}%,
}

\begin{document}

\title{A logical analysis of instrumentality judgments:\\
means-end relations in the context of experience and expectations}

\author{Kees van Berkel%
  \thanks{Electronic address: \texttt{kees@logic.at}}}
  \affil{Ruhr Universit\"at Bochum, Institute for Philosophy II, Universit\"atsstra{\ss}e 150, 44801 Bochum, Germany}

\author{Tim S. Lyon%
  \thanks{Electronic address: \texttt{timothy\_stephen.lyon@tu-dresden.de}}}
\affil{Technische Universit\"at Dresden, Institute of Artificial Intelligence, N\"othnitzer Str. 46, 01069 Dresden, Germany}

\author{Matteo Pascucci%
  \thanks{Electronic address: \texttt{matteopascucci.academia@gmail.com}}}
\affil{Institute of Philosophy, Slovak Academy of Sciences, Klemensova 19, 811 09 Bratislava, Slovak Republic}

\date{}

\maketitle

\begin{abstract}
\noindent  This article proposes the use of temporal logic for an analysis of instrumentality inspired by the work of G.H. von Wright. The first part of the article contains the philosophical foundations. We discuss von Wright's general theory of agency and his account of instrumentality. Moreover, we propose several refinements to this framework via rigorous definitions of the core notions involved. In the second part, we develop a logical system called Temporal Logic of Action and Expectations ($\lae$). The logic is inspired by a fragment of propositional dynamic logic based on indeterministic time. The system is proven to be weakly complete relative to its given %Kripke-style
semantics. %axiomatisation.
We then employ $\lae$ to formalise and analyse the instrumentality relations defined in the first part of the paper. Last, we point out philosophical implications and possible extensions of our work.  
%and suggest ways in which our formalism could be employed in deontic settings.
\end{abstract}

\noindent \textit{Keywords}: Action logic $\cdot$ Instrumentality $\cdot$ Means-end relation $\cdot$ Philosophical logic 
$\cdot$ Temporal logic $\cdot$ Von Wright

\section{Introduction}\label{Sec:introduction}

Agents shape their world by making choices, performing actions, and exercising abilities. They reason practically about attaining ends, plan short-term and long-term, and comply with and violate norms. Agents may be right, lucky, and mistaken in their planning. In all of the above, instrumentality statements---also referred to as means-end relations---play a vital role. They provide reasons to act in one way instead of the other. For instance, the statement `Taking the A-train is an excellent means for going to Harlem’ may influence whether I visit my friend Eduard, who lives in Harlem, by train. My experience with traffic jams in New York may cause me to refrain from going by car instead. 

The concepts of action, ability, and choice are central to any theory of agency \cite{Ans00,Dav16,Gol70,Har71,Wri63} and have been extensively investigated in philosophical logic \cite{Aqv02,BelPerXu01,Bro88,Seg92,Wri68}. 
In contrast, the philosophical and logical investigation of instrumentality relations has thus far received comparably little attention in the literature. In particular, the formal study of how agents acquire and comparatively assess the quality of instrumentality judgments remains to be conducted. This is especially noteworthy due to the role of instrumentality statements in the fields of practical reasoning \cite{Aud89,Cla87,Har71,Raz78,Wri63b,Wri72b}, AI planning \cite{Meyetal15,RaoGeo95}, and linguistics \cite{ConLau16,Sae01,vonSte06}. Although it is common to analyse reasoning \textit{with} such statements, none of the above accounts treat \textit{how} various instrumentality judgments are obtained, compared, and assessed. To our knowledge, Georg Henrik von Wright is the only philosopher who provides such a philosophical, yet brief, account of instrumentality judgments. This article provides a formal account of instrumentality inspired by von Wright's philosophy.
  
Consider the following practical problem:
\begin{quote}
I would like to have this small parcel opened. What should I do? Which action is the best choice for securing my desired end? For instance, would ripping the parcel's cardboard, cutting the tape with a knife, or using a pair of scissors be most suitable for my purpose?  
\end{quote}
To satisfactorily address the above problem, I must find out which actions are (most) suitable for attaining my goal. Various challenges arise: I need to somehow collect candidate instruments that will satisfactorily resolve my practical problem. Subsequently, none of the instruments at my disposal may be necessary, that is, there are only sufficient means available. How should candidate means be compared? Which must be preferred and chosen? Is there a difference in the quality of these available instruments?

In the aforementioned fields, instrumentality relations are commonly taken as given. %Furthermore, %only necessary means are considered or priority orderings over sufficient means are assumed. In particular,
For instance, most accounts of practical reasoning limit illustrations of practical inference to cases of necessary means, yielding oversimplified representations % by limiting practical inferences to cases of necessary means 
\cite{Har71}. %, whereas such reasoning is primarily about sufficient means. 
Exceptions are \cite{Cla87,Lew17,Wal07} which discuss versions of comparitivism as a way to resolve cases in which various sufficient instruments are available. Nonetheless, in all these accounts instrumentality relations are assumed to be present from the outset, bypassing the following questions: %  We, therefore, ask the following questions, which are of primary interest in this paper:
  \begin{itemize}
\item[(q1)]    \textit{How are instrumentality judgments acquired by an agent?}

\item[(q2)] \textit{How can the comparative quality be assessed for such judgments?}
  \end{itemize}

This article addresses questions q1 and q2. Concerning the above practical problem, one may respond to these questions in various ways. For instance, one may invoke experience, apply theoretical knowledge, or follow the advice of a trustful adviser. Given my desire to open the parcel, I may recall from personal experience that using a knife or a pair of scissors always sufficed. %to open a parcel in the past.
Then again, I may recall a piece of advice that one should avoid using knives to open parcels since they may damage the contents. Alternatively, I may search the internet for possible ways to open my parcel. In this article, we propose several ways to compare actions as instruments for a given purpose. In particular, we assign a pivotal role to the agent’s past in yielding instrumentality judgments. 

\paragraph{Three Criteria.}
Guided by questions q1 and q2, any account of instrumentality should be able to address the following three points: First (I),\label{criterionI} the account must clarify what it means to say that for an agent $\alpha$, action $\Delta$ is an instrument serving purpose $\phi$. The fact that instrumentality is a relative notion becomes clear when we see that different actions may serve different ends for different agents with different abilities. Considering the previous example, I might not be comfortable with knives which means that using a knife to open a parcel would not be a suitable instrument for me (although it might be for someone else). 

Second (II),\label{criterionII} the account must not only provide procedures to determine which instruments are suitable for the purpose at hand, but it must also allow for a comparison of the different instruments collected. The assessment of instrumentality judgments is \textit{axiological} (from the Greek `ax\'ia', for `value'), i.e., it provides a label expressing a specific value of the instruments at hand; e.g., scissors are \textit{good} instruments for opening parcels, although using your hands would be \textit{better}. Eventually, such axiological judgments concerning instruments may serve as a guide in practical decision-making.

As a third point (III),\label{criterionIII} we observe that one of the typical features of instrumentality relations is that they are not given once and for all. In many cases, such judgments are established via inductive arguments, relying on past experience witnessing the connection between the terms involved. As famously noted by David Hume \cite{Hum39}, inductive arguments are affected by a fundamental problem: how are we justified in making inferences from an observed connection in the past to instances of that connection of which we have no experience? For example, in forming instrumentality judgments based on experience, an individual agent can often not collect all relevant past cases that would settle the issue. Even then, future cases may still be different. For this reason, judgments of instrumentality are essentially \textit{defeasible}. For instance, such judgments may need to be revised due to new personal experience, additionally secured information on the past, and newly received data from other agents.

\paragraph{Contributions.} 
In this work, we address the above three criteria. We do this by developing a formal account of instrumentality. Our account is grounded in von Wright's theory of agency in general and his analysis of instrumentality in particular. In the work `The Variety of Goodness' \cite{Wri72}, von Wright provides a philosophical discussion of instrumentality by which actions can be judged as `good instruments' for specific purposes.

In order to address (I)-(III), we take von Wright's account as a departure point and extend it where necessary. Those parts rooted in von Wright's account are made explicit throughout this work. We define a logical system called the \underline{T}emporal \underline{L}ogic of \underline{A}ctions and \underline{E}xpectations ($\lae$) that will formally capture the developed theory of instrumentality. We prove the weak completeness of $\lae$ with respect to a corresponding Kripke-style semantics using a variation of the Fischer-Ladner construction for propositional dynamic logic \cite{FisLad79}. Employing $\lae$, we provide formalisations of the acquired conceptions of (comparative) instrumentality and discuss the philosophical implications of our formal setting.
  
\paragraph{Outline.} %The article is structured as follows: 
Section \ref{Sec:philosophy_agency} consists of a brief analysis of von Wright's general theory of agency and his theory of instrumental goodness. We refine von Wright's theory by supplementing criteria for comparing instruments that serve the same purpose. In particular, in \sect~\ref{subsect:philo_good} and \ref{subsect:philo_better}, we address criterion I by providing various notions of instrumentality deduced from an agent's past experience. In \sect~\ref{subsect:philo_comparing}, we deal with criterion II by providing different ways of comparing instrumentality judgments, yielding various value judgments. \sect~\ref{subsect:philo_expectations} is devoted to the defeasibility of instrumentality judgments as expressed in criterion III, and additionally contains a refinement of our take on criterion I. In Section \ref{Sec:logic_lae}, we discuss the requirements of our formal language and introduce the logical system $\lae$, which %. The resulting logic represents 
is a temporal extension of the logic `Logic of Actions and Expectations’ ($\mathsf{LAE}$) in %as developed in
\cite{BerPas18}. In Section \ref{Sec:weak_completeness}, we prove that $\lae$ is weakly complete. After that, in Section \ref{Sec:formal_instrumentality}, we discuss the philosophical implications of our formalism and address criteria I-III formally: we logically formalise different notions of instrumentality (criterion I), present several semantic notions of comparative instrumentality (criterion II), and discuss the defeasible nature of instrumentality statements (criterion III). In Section \ref{sec:closing_remarks}, we address future work.

\section{Agency and Instrumentality}\label{Sec:A}\label{Sec:philosophy_agency}

The backbone of our philosophical analysis will be von Wright's general theory of agency, as laid out in \cite{Wri63,Wri68,Wri72}. We start with a brief analysis of the theory and refer to \cite{Aqv02,BerPas18,Sto10} for a more extensive discussion. %\footnote{Von Wright is considered one of the founders of the logic of action \cite{Aqv02,Seg92}.} 
 We subsequently provide a philosophical analysis of instrumental goodness and comparative instrumentality. The former is primarily based on von Wright's discussion of instrumental goodness as laid out in~\cite[pp.~19-40]{Wri72}. In short, instrumental goodness covers the study of judgments concerning how \textit{well} instruments serve their purpose (equivalently, how well means serve their ends). Since von Wright's analysis is sparse---i.e., being a sub-topic of his theory of goodness \cite{Wri72}---we extend his account in two ways: (i) we discuss several refined instrumentality notions and (ii) consider possible definitions of comparing instruments. The theory presented in this section will ground the subsequent formalisation. 
 %logical formalisation provided later in the paper.

\subsection{Von Wright's General Theory of Agency}

According to von Wright, to act is to interfere with the course of nature~\cite{Wri63}. Such interference manifests itself in bringing something about or preventing something from happening. What is brought about or prevented is a particular \textit{state of affairs}, i.e., a partial description of the world such as `the parcel is open'. To bring about a result $\phi$ means to act ``in such a manner that the state of affairs that $\phi$ is the result of one's action'' \cite[p.~13]{Wri63}. Likewise, prevention of $\phi$ indicates that one's action has succeeded in ensuring $\lnot \phi$.

\begin{figure}
    \centering
\begin{tikzpicture}
    \node[world] (w0) [label=right:{$\lnot \phi$}] {$w_0$};
    \node[world] (w1) [above=of w0, xshift=-25pt, label=above:{$\phi$}] {$w_1$};
    \node[world] (w2) [above=of w0, xshift=25pt, label=above:{$\lnot \phi$}] {$w_2$};
    \node[] (label) [below=of w0, yshift=25pt] {(i) produce $\phi$};
    \path[->, draw] (w0) -- (w1) node [midway, xshift=7pt] {$a$};
    \path[->, draw] (w0) -- (w2) node [midway, xshift=7pt] {$n$};
    \end{tikzpicture}
    \hspace{0.4cm}    
\begin{tikzpicture}
    \node[world] (w0) [label=right:{$\phi$}] {$w_0$};
    \node[world] (w1) [above=of w0, xshift=-25pt, label=above:{$\lnot \phi$}] {$w_1$};
    \node[world] (w2) [above=of w0, xshift=25pt, label=above:{$\phi$}] {$w_2$};
    \node[] (label) [below=of w0, yshift=25pt] {(ii) destroy $\phi$};
    \path[->, draw] (w0) -- (w1) node [midway, xshift=7pt] {$a$};
    \path[->, draw] (w0) -- (w2) node [midway, xshift=7pt] {$n$};
    \end{tikzpicture}
        \hspace{0.4cm}      
\begin{tikzpicture}
    \node[world] (w0) [label=right:{$\lnot \phi$}] {$w_0$};
    \node[world] (w1) [above=of w0, xshift=-25pt, label=above:{$\lnot \phi$}] {$w_1$};
    \node[world] (w2) [above=of w0, xshift=25pt, label=above:{$\phi$}] {$w_2$};
    \node[] (label) [below=of w0, yshift=25pt] {(iii) suppress $\phi$};
    \path[->, draw] (w0) -- (w1) node [midway, xshift=7pt] {$a$};
    \path[->, draw] (w0) -- (w2) node [midway, xshift=7pt] {$n$};
    \end{tikzpicture}
        \hspace{0.4cm}      
\begin{tikzpicture}
    \node[world] (w0) [label=right:{$\phi$}] {$w_0$};
    \node[world] (w1) [above=of w0, xshift=-25pt, label=above:{$\phi$}] {$w_1$};
    \node[world] (w2) [above=of w0, xshift=25pt, label=above:{$\lnot \phi$}] {$w_2$};
    \node[] (label) [below=of w0, yshift=25pt] {(iv) preserve $\phi$};
    \path[->, draw] (w0) -- (w1) node [midway, xshift=7pt] {$a$};
    \path[->, draw] (w0) -- (w2) node [midway, xshift=7pt] {$n$};
    \end{tikzpicture}
    \caption{Von Wright's four elementary types of action. The transition (arrow) from $w_0$ denoting the agent acting is labelled $a$, and the alternative transition indicating the agent's non-interference with nature is labelled $n$.}
   
\label{fig:VW_four_types_of_action}
\end{figure}
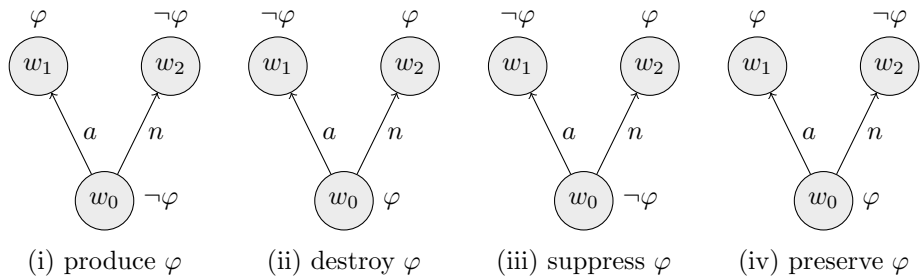

 The above concept of action is founded on the notion of \textit{change}. In fact, for von Wright, any theory of agency and action must presuppose an account of change, e.g., see \cite{Aqv02}. He takes change to define a \textit{transition} from an initial state (i.e., the present moment) to an end-state (i.e., a future moment). Such transitions can be either agent-independent (e.g., a moon eclipse) or agent-dependent (e.g., me cutting the tape of my parcel). The agent-dependent setting forces a non-deterministic worldview. That is, to bring something about forces at least the following three elements: the initial state, in which the agent finds herself, the actual end-state (which is the state that emerges after the performed action), and an alternative end-state (which would
 result if the agent refrains from performing the action).
  
Von Wright discusses various relations between these three states. By bringing together the above account of change with the twofold distinction of bringing about and prevention, he characterises four types of action: \textit{producing}, \textit{destroying}, \textit{suppressing}, \textit{preserving}. The first two bring about something, whereas the latter two prevent something. (We refer to \cite{Aqv02,BerPas18} for a discussion of these action types in a formal setting.) The four action types are characterised in \fig~\ref{fig:VW_four_types_of_action}. 
 For instance, at (iii), the act of suppressing $\phi$ indicates that at the initial state $\lnot \phi$ holds, through the agent's acting $\lnot \phi$ continues to hold, and if the agent would not have acted $\phi$ would have come about. 

Von Wright's
reading of the four action types is arguably too strong:  
%since it overlooks the \textit{uncertainty} of action: 
the agent's acting in \fig~\ref{fig:VW_four_types_of_action} ultimately decides the fate of $\phi$. In the case of producing, by acting, the agent ensures $\phi$, whereas, by not-acting, the agent can ensure $\lnot \phi$. In other words, von Wright's account takes agency as causally sufficient in both directions, e.g., see \cite{Aqv02}. Furthermore, in \fig~\ref{fig:VW_four_types_of_action}, there is no distinction made between different kinds of action an agent can perform at $w_0$.
  
Since a general analysis of agency involves many distinct agents and distinct agents can simultaneously perform distinct actions, an individual agent's action is often not causally sufficient. This is known as the \textit{uncertainty} of action \cite{BelPerXu01}. We adopt a generalisation of von Wright's approach that includes this uncertainty. A transition involves the following three elements: (i) an initial state, (ii) a \textit{set} of actions, and (iii) a \textit{set} of possible final states. Henceforth, we also refer to such states as \textit{moments} in indeterministic time. A single agent does not entirely control the course of events, but even the complete set of actions performed by all agents involved does not necessarily entail a unique end-state (cf. the influence of nature). For this reason, we say that a set of actions \textit{causally contributes to the attainment of} the (actual) end-state of a transition only if the end-state would have been different without the performance of that set of actions.

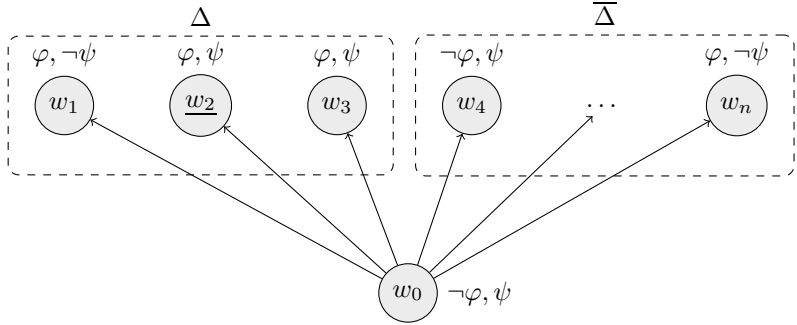
\begin{figure}
    \centering
   
\begin{tikzpicture}
    \node[world] (w0) [label=right:{$\lnot \phi, \psi$}] {$w_0$};
%    \node[] (wp) [below=of w0] {$\dots$};
    \node[world] (w1) [above=of w0, xshift=-130pt, yshift=20pt, label=above:{$\phi, \lnot \psi$}] {$w_1$};
    \node[world] (w2) [right=of w1, label=above:{$\phi, \psi$}] {$\underline{w_2}$};
    \node[world] (w3) [right=of w2, label=above:{$\phi, \psi$}] {$w_3$};
    \node[world] (w4) [right=of w3, label=above:{$\lnot \phi, \psi$}] {$w_4$};
    \node[] (wdot) [right=of w4] {$\dots$};
    \node[world] (wn) [right=of wdot, label=above:{$\phi, \lnot \psi$}] {$w_n$};
    \node[draw, dashed, color=black, rounded corners, inner xsep=10pt, inner ysep=15pt, fit=(w1)(w3), label={$\Delta$}] (c1) {};
    \node[draw, dashed, color=black, rounded corners, inner xsep=10pt, inner ysep=15pt, fit=(w4)(wn), label=$\overline{\Delta}$] (c2) {};
  %  \path[->,draw] (wp) to (w0);
    \path[->,draw] (w0) to (w1);
    \path[->,draw] (w0) to (w2);
    \path[->,draw] (w0) to (w3);
    \path[->,draw] (w0) to (w4);
    \path[->,draw] (w0) to (wdot);
    \path[->,draw] (w0) to (wn);
    \end{tikzpicture}
       
\caption{Indeterministic time and uncertainty of action: producing $\phi$ through performing action $\Delta$ at $w_0$. Let $\phi$ stand for `the parcel is open', $\psi$ for `the parcel is undamaged', and $\Delta$ for `using a knife to cut the parcel's tape'. Alternatively, one may label arrows to denote $\Delta$-transitions (cf. \fig~\ref{fig:expectations_generalizations}).}
   
\label{fig:indeterministic tree}
\end{figure}

For instance, considering \fig~\ref{fig:indeterministic tree}, we say that the action $\Delta$ brings the agent from the present moment $w_0$ to either one of the future moments $w_1$, $w_2$, or $w_3$ without strictly determining either of the three. Still, all three moments satisfy $\phi$. Thus, we say that the agent can produce $\phi$ by performing $\Delta$ at $w_0$, even though the agent cannot secure a unique future moment with action $\Delta$ (e.g., one could read $\phi$ as `the parcel is open' and $\Delta$ as `using a knife to cut the parcel's tape'). Furthermore, performing $\overline{\Delta}$---which is the complement of $\Delta$---could lead to at least one future moment where $\lnot \phi$ holds, namely, $w_4$. Suppose that $w_2$ is the actual future moment (underlined in \fig~\ref{fig:indeterministic tree}). In that case, the agent causally contributed to the attainment of $w_2$ through performing $\Delta$, since if the agent had refrained from performing $\Delta$, the resulting possible future moments would have been one of $w_4,\ldots,w_n$. Last, although $\Delta$ successfully ensures $\phi$, the action can still fail to secure other ends such as $\psi$ in \fig~\ref{fig:indeterministic tree} (e.g., where $\psi$ reads `the parcel is undamaged'). 

As a final remark concerning the nature of the term `action',
we will follow the usual distinction between types (i.e., generic categories, such as `writing') and tokens (i.e., concrete instances in specific circumstances, such as the action of a particular person writing on a particular blackboard on a specific date; see \cite{Gol70} for an extensive discussion). Von Wright adopts a similar demarcation by distinguishing between
act-categories, on the one hand, and
act-individuals, on the other hand \cite[p.~36]{Wri63}.
It suffices to restrict our analysis to 
{\textit{atomic} actions}, \textit{negative} actions (e.g., `not crossing the street') and
{\textit{complex}} actions (e.g.,  `turning left or turning right' and `turning left and hitting the break'). We do not consider \textit{sequences} of actions.

\subsection{Instrumentality and Instrumental Goodness}\label{subsect:philo_good}

We are now in a position to address criterion I of \sect~\ref{Sec:introduction}, and clarify what it means that an action $\Delta$ is an instrument serving purpose $\phi$. Our analysis will yield two provisional definitions at the end of the next subsection. Central to the study of instrumentality is the relation between an \textit{action} and a \textit{result}. The former is the instrument for the desired outcome expressed in the latter. The outcome can therefore be seen as the purpose of performing the action in question. Thus, we refer to the action as an \textit{instrument} serving a particular \textit{purpose}. For example, `pulling down the lever of a door and drawing the door towards you' is the instrument for the result `the door is open'. 
%An alternative way of referring to this relation is through a \textit{means-end} relation.

Following von Wright \cite[p.~21]
{Wri72}, we can group actions into categories, or kinds, according to the purposes served. To illustrate, no action will qualify as a \textit{cutting}-instrument unless it can serve the purpose of cutting. 
In this respect, the ability to serve a `cutting purpose' is a functional characteristic of members of the kind `cutting'(-instruments). Actions do not necessarily serve a unique purpose. As an instrument, an action can be a member of several kinds, serving multiple purposes. The action `using a knife' is an instrument for opening parcels and for peeling apples. 

The \textit{goodness} of an instrument is judged in relation to \textit{how} it serves the purpose with which it is associated. Consequently, an instrument may serve certain purposes with excellence while serving others poorly. 
In this article, 
we concentrate on a specific agentive source for judgments of instrumental goodness: the agent's (personal) experience with the instruments serving the purpose. In the case of the agent's experience, %the same strategy for ascertaining judgments of instrumentality may be employed: 
the agent checks her past for applications of those actions belonging to the same type as the one under consideration, checking whether these past applications successfully served the purpose at hand. This temporal component will be central to our definition of instrumental goodness.

We point out that for instrumentality judgments what matters is an agent's \emph{perception} of the relation between an action and effect. That is, we are concerned with what the agent believes is the relation between an instrument and purpose, which may or may not be rooted in any physical causal relation. For instance, to form a judgement concerning `cutting knives' and `opening parcels', an agent needs no (prior) knowledge of the physics involved in cutting the tape of a parcel (the sharpness of the blade, the movement of the arm, the consistency of the tape, %the humidity of the air,
etc.). In our case, what counts is the agent's (past) experience of the event (opening a parcel with a knife) and the beliefs that were formed accordingly. Therefore, when we say that a knife is an instrument for opening parcels, we do not refer to its set of physical causal qualities \textit{per se}. In our opinion, the notion of an instrument is a \textit{practical} one, and it would generally be better to keep it distinct from the notion of a cause, as employed in the analysis of \textit{physical} connections between things.\footnote{This relates to Hume's criticism of the notion of causation: causality judgements depend on the set of beliefs of a specific group of agents given a specific context and cannot be generalised to constitute physical laws. They are not intended to have such a universal validity.} We see past experience as a fruitful approach to agentive reasoning since it is a source accessible to the agent at any given time.\footnote{Past experience is not exhaustive. Additionally, one may consider here-say, advice, education, and theoretical knowledge as sources alternative to past experience. Such alternative sources fulfil an important role in accounting for the employment of instruments that have never been used before. We leave this to future work.}

This article further develops the idea of rooting instrumentality judgments in past experience. Von Wright involves past experience in his analysis of goodness of skill and instrument, but the expansion of the idea as presented in this article is our own. We propose that instrumentality judgements of this type are formed in \textit{three steps}. In the first step, the agent collects all relevant evidence available from the past. This is the \textit{empirical part} of the procedure. The second step consists in forming context-relative judgments based on what past experience has taught. For instance, in considering all cases in which two actions were performed, it may turn out that a particular outcome was more frequently obtained in association with one of the actions than with the other. This is the \textit{inductive part} of the procedure (i.e., the generalisation of the relevant experience). Third, inferences can be drawn from these context-relative judgements in combination with considerations of chance, (unpredictable) interference, and additional circumstances. Such inferences serve as a guide for the agent's current behaviour. This is the \textit{deductive part} of the argument.

\subsection{Three Notions of Instrumentality}\label{subsect:philo_better}

The ambivalent nature of the philosophical concept of `good' may suggest that judgments of instrumental goodness are not objective. 
However, von Wright \cite[p.~25]{Wri72} argues that in dealing with instrumental goodness, the relation between `purpose’, `instrument’ and `good’ is, in fact, objective. Namely, judgments of instrumental goodness express two types of `connexion'. The first is a causal connexion, expressing a relation between the instrument (as a perceived cause) and its purpose (the desired effect). Von Wright emphasises that this connexion can be empirically checked and is thus objective. Extending the above to cover past experience, we observe that collecting an agent's relevant past experience is also an empirical process, having the same objectivity status. To avoid confusion, we refer to the first connexion as an \textit{empirical connexion} instead. The second is a \textit{logical} connexion that holds between the purpose in question and the term `better'. Namely, given the agent's desire for a particular effect (the purpose), one can order the available actions according to their empirical connexion such that one action can be logically determined to be a better instrument than another. Accordingly, judgments of instrumental goodness can be objective since we refer to an observable, empirical property of the instrument and assign goodness based on an ordering, which is a judgment of logic.

There are several ways to assess the goodness of an instrument. For von Wright, `goodness' always refers to particular observable properties of the instrument that make it suitable for some purpose. This property is called the good-making property. Good-making properties are the properties of the instruments that enable the logical ordering of goodness. To illustrate, if one intends to have some chopped vegetables (the purpose), the good-making property `sharper’ can determine which of the available knives (instruments) is better. Here, sharpness determines the causal connexion between the knife and cutting. One can then order all available knives according to their sharpness to determine which are sharper and, thus, better \cite[p.~25]{Wri72}.

This article takes past experience as a general property for ordering instruments. In particular, we focus on the successful applications of the instrument that guaranteed the effect, as witnessed by the agent. Given this good-making property, a judgment about one instrument `being better’ than another is a logical consequence of ‘being more successful at guaranteeing the desired result’.

To qualify a generic instrument---i.e., action type---as good for some purpose, we need to ground our judgment in the past performance of concrete actions, i.e., action tokens. We refer to this \textit{temporal component} as the historical witness of an instrument’s suitability. In inquiring about whether to apply a certain instrument, agents often base decisions on statements such as:
\begin{itemize}
    \item[(i)] `it has worked before';
    \item[(ii)] `so far, it has not disappointed me';
    \item[(iii)] `well, it has thus far worked better than any of the alternatives'.
\end{itemize}
These three remarks illustrate the importance of temporal reference. Furthermore, we identify three different criteria of instrumental goodness in (i)-(iii): The first remark exemplifies, what we will call, the minimum criterion for any appropriate definition of instrumentality: (i) the action \textit{has} served the purpose at least once and, for that reason, it \textit{can} serve the purpose as an instrument. That is to say, criterion (i) functions as a lower bound on the instrument's suitability, thus identifying \textit{candidate instruments}. In the second remark, we recognise an upper bound, i.e., a maximum criterion: (ii) there have been applications of the instrument and these applications \textit{have always} served the purpose. Criterion (ii) is referred to as instrumental \textit{excellence}. In the last remark, we identify a \textit{comparative} approach to instrumental goodness: (iii) the action is suitable in \textit{comparison} to alternative actions.

Observe that (i) and (ii) express, respectively, an existential criterion and a universal (constructive) criterion. These criteria enable us to label instruments as `good', independent of their relation to other instruments. The third (iii) interpretation is a comparative notion that labels instruments as good but only relative to other instruments. If an instrument is not excellent in itself but is the best among others, we refer to the instrument as \textit{comparatively excellent}.

We make our analysis precise in \dfn~\ref{def:candidate_instruments}. In what follows, we use $\Delta$, $\Gamma$, $\Sigma, \dots$ to denote action types, $\alpha$, $\alpha_1, \alpha_2 \dots$ to denote agents, $\phi$, $\psi$, $\chi, \dots$ to denote propositions (i.e., descriptions of states of affairs), and $w,v,u,\dots$ to denote moments. The terms `moment' and `state' are used interchangeably and the expression `$\phi$-instrument' abbreviates `instrument to obtain the state of affairs described by $\phi$'. %We propose the following definition of \textit{candidate} instruments:

\begin{definition}[Definition of Candidate Instruments]\label{def:candidate_instruments} \ 

\begin{itemize}
\item[(1)] \textsc{
candidate
instruments:} An action type $\Delta$ is a candidate $\phi$-instrument for agent $\alpha$ at moment $w$
if and only if (i) $\Delta$ has led to $\phi$ for $\alpha$ at least once in the past of $w$.
\item[(2)] \textsc{excellent candidate instruments:} An action type $\Delta$ is an excellent
candidate 
$\phi$-instrument for agent $\alpha$ at moment $w$
if and only if (i) $\Delta$ is a candidate $\phi$-instrument for $\alpha$ at $w$ and (ii) $\Delta$ has always led to $\phi$ for $\alpha$ in the past of $w$.
\item[(3)] \textsc{better candidate instruments:} An action type $\Delta$ is a better candidate $\phi$-instrument for agent $\alpha$ at moment $w$
than the candidate $\phi$-instruments $\Gamma_1,\ldots,\Gamma_n$ available to $\alpha$ at $w$ if and only if (i) $\Delta$ is a candidate $\phi$-instrument for $\alpha$ at $w$ and (ii) in the past of $w$, action $\Delta$ was more successful for $\alpha$ in guaranteeing $\phi$ 
%has led to $\phi$ more frequently for $\alpha$ 
than $\Gamma_1,\ldots,\Gamma_n$.
\end{itemize}
\end{definition}

We point out that judgments of the type expressed above vary over agents, and since the qualification of instruments is defined relative to the past, these judgments are strictly dependent on a vantage point too. This makes the above three definitions %ions of instrumental goodness 
\textit{defeasible}: new experience may cause the agent to revise previous judgments. For example, horses may have been the best option for private transport before cars. We discuss defeasibility in detail in Section~ \ref{subsect:philo_expectations}.

\subsection{Different Ways of Comparing}\label{subsect:philo_comparing}

In \dfn~\ref{def:candidate_instruments}, we defined the notion of a `better instrument' in terms of the \textit{relative success} of the instrument in question. In what follows, we address criterion II of \sect~\ref{Sec:introduction} and discuss various ways of comparing the success of candidate instruments. The analysis will yield different definitions of comparative instrumental goodness.

Von Wright does not provide an explicit method for comparing candidate instruments, but we find some hints in his analysis of technical goodness (the goodness of ability, capacity, and skill) \cite{Wri72}.\footnote{Von Wright \cite%[Ch.~2-3]
{Wri72} explicitly distinguishes \textit{technical} goodness from \textit{instrumental} goodness. The former relates to agents' skills in obtaining results and applying instruments. The latter denotes an impersonal analysis of how instruments serve purposes. We do not differentiate between instrumental and technical goodness, but the logic developed in this paper does allow for reasoning with agent-dependent and -independent notions of instrumentality.% This captures, to some degree, von Wright's initial distinction.
} There are two ways of testing whether an agent excels at guaranteeing a particular result or at performing some action, namely, (i) through \textit{competition} and (ii) through \textit{achievement} \cite[Ch.~2]{Wri72}. Method (i) evaluates personal performance in relation to other agents. In the context of the Olympic games, competition is a means to determine which of all candidate athletes excels at, say, the javelin throw. Such competition can be defined in terms of `who comes in first' in a single game (the absolute best) or in terms of who ended first in a sequence of games (the overall best). Method (ii) evaluates an agent's performance of an action, not in light of other agents acting, but in relation to a predetermined threshold of excellence. Achievement may be defined as beating a predetermined time or record (which may or may not be previously set by other agents) or achieving a qualifying threshold. % that determines eligibility for participating in the Olympic games. 
 Such thresholds function as markers of a specific grade of excellence.

Following von Wright \cite[p.~34]{Wri72}, both (i) and (ii) are methods of establishing goodness by `degree of distinction’, thus constituting a comparison. Method (i) hints at judging goodness on the basis of ordering the success ratios of participating candidates, whereas method (ii) hints at the usage of pre-established thresholds that suitable candidates must pass.

Translating the above to %the setting of
instrumental goodness, we find two notions of comparison. The first compares a candidate instrument to alternative candidates, which calls for an ordering of (all) suitable instruments serving the purpose at hand. First, we gather all instruments that potentially deserve the label `good-instrument', these are candidate instruments (see \dfn~\ref{def:candidate_instruments}). Then, for each candidate, we collect its successful and unsuccessful applications % the action to 
in attaining the desired outcome. 
Thus, we construct a success-failure ratio for each candidate instrument. 
An instrument comparatively excels in serving a particular purpose if it has the highest success ratio among its alternatives. 

However, we point out that this ordering might fail to properly assess whether an instrumentality relation exists between $\phi$ and available actions $\Gamma_1,\dots,\Gamma_n$. This can happen when some of the considered actions were performed  only a small number of times within the considered time. A very small sample size may have no effective statistical power since the observed connections between $\phi$ and $\Gamma,\dots,\Gamma_n$ could have resulted from mere chance. In such cases, it is reasonable to set thresholds such that one can avoid engaging in a judgment of those actions that have not been sufficiently tested.

The second method considers a single candidate instrument and checks whether it has satisfied a certain threshold. We identify two types of threshold. The first consists in setting a minimum threshold of the success-failure ratio. For instance, a candidate $\phi$-instrument is a good $\phi$-instrument if it guarantees the effect $\phi$ at least half of the time. Although such a threshold is reasonable, there are problems with it. Namely, suppose I throw a dart at the bullseye and end up hitting the bullseye during my first throw (without any practice). Naively, for me throwing darts is an excellent instrument for hitting the bullseye since, thus far, I have always been successful.  
Often, it makes sense to require an additional threshold in the light of which the candidate instrument must be evaluated. This is the second type: we may require that the instrument has been at least applied $n$-many times in attempting to attain the end in question (e.g., think of products that must be tested before they can be sold).

Likewise, time plays a role in setting thresholds: imposing a threshold on the number of applications entails imposing a minimum length on the time interval considered. %to collect data.
One needs to ensure that the time interval allows for a number of action performances that is at least equal to the threshold. Setting such a limit has another function: it excludes cases that lie too far in the past. For instance, in determining whether I excel at hitting the bullseye it suffices to consider, say, my last 100 attempts. Without a limit, past cases when I was still learning how to play darts would be included, thus misrepresenting my current skills.

Based on the above, we propose the following two definitions of comparison. 
\begin{definition}[Notions of Comparison]\label{def:comparison_notions} A candidate $\phi$-instrument $\Delta$ can be evaluated with respect to:
\begin{enumerate}
\item  other candidate $\phi$-instruments $\Gamma_1,\ldots,\Gamma_n$ based on a success-failure ratio.
\item a  threshold:

\subitem (i) on a lower bound of past co-occurrences of $\Delta$ and $\phi$;

\subitem (ii) on a lower bound of the ratio of past co-occurrences of $\Delta$ and $\phi$ versus past occurrences of $\Delta$ without $\phi$;

\subitem (iii) on both (i) and (ii).
\end{enumerate}
\end{definition}

The above % twofold
distinction gives rise to two types of comparative instrumental goodness: (1) goodness in relation to other instruments and (2) goodness in relation to a set threshold. Combinations of (1) and (2) are also possible: e.g., consider the context of the Olympic games, where an athlete must acquire a certain amount of credits in order to enter the games initially (threshold). 
An advantage of adopting approach (1) is that it completely preserves the logical status of comparative goodness by ordering available instruments based on their success ratio, independent of any external measurement such as a threshold. The downside of (1) is that it allows for cases such as `having a lottery ticket is an excellent instrument for winning the lottery since it is the only way to win (despite not being sufficient)'. % is a necessary precondition'.} 
In such cases, a certain threshold is desirable.\footnote{There is another issue concerning the lottery example. Since a myriad of tickets is available for any single lottery draw, and tickets are assigned to buyers randomly, the win-loss proportion should be based on the number of winning tickets vs the number of total tickets. In such cases, reference to an agent is not a primary parameter. The issue relates to the low probability of the instrument serving the purpose. This indicates that some actions may not be suitable as instruments at all. 
%That is, due to the involvement of probabilities, participating in a lottery 
$n$ times and winning $n$ times will not make it an excellent instrument. 
There is no sufficient instrument for winning the lottery (perhaps except by buying all the tickets) since winning is beyond the agent's abilities.}

Following von Wright, we can label instruments as `better', `best', and `poorest' by ordering the available instruments. The best instrument will be the instrument which is unsurpassed in its success ratio by any other, and the poorest will be unsurpassed in its low success ratio. A poor instrument, however, is not necessarily a bad instrument \cite[p.~35]{Wri72}. The former does not serve the purpose well, whereas the latter may serve the purpose well but have (legally or socially) undesirable side effects (e.g., think of `opening a parcel by first stealing the knife with which you intend to open the parcel'). %In \sect~\ref{Sec:formal_instrumentality}, we consider formalizations and expansions of the various types of goodness discussed. 

\subsection{Expectations: The Other Temporal Component}\label{subsect:philo_expectations}

So far, we discussed how the past serves as a fruitful source of information for qualifying relations of instrumentality. Judgments of instrumentality are essential for practical reasoning since the latter concerns how states of affairs can be altered through the intervention of an agent (cf. \cite{Aud89,Cla87,Har71,Raz78}). Practical reasoning is, thus, essentially directed towards the \textit{future}. Through instrumentality judgments, agents may generalise past experience and project it onto the future. This generalisation and projection lies at the heart of the problem of induction. We now address criterion III of \sect~\ref{Sec:introduction}.

In \cite{Wri57}, von Wright deals with the problem of induction. In particular, von Wright divides the induction problem into two parts (see \cite[p.~50]{Wri57}):
\begin{itemize}
\item[(q3)]  How can we demonstrate that the generalisations we make about experienced cases are correct? How do we know that our generalisation is `complete'?
\item[(q4)] How can we demonstrate that such generalisations are reliable for making predictions? That is, how can we extend our generalisations to the future?
\end{itemize}
Interestingly, von Wright's division is temporal: (q3) deals with the past and (q4) with the future. Regarding generalisations extending to the past, one can theoretically acquire universally objective judgments by collecting all past instances of the object under generalisation. However, when it comes to predictions, the problem of induction truly shows itself:
\begin{quote}``Scarcely anybody would pretend that predictions, even when based upon the safest inductions, might not fail sometimes.''~\cite[p.~51]{Wri57}
\end{quote}
Regarding the future, generalisations are inherently defeasible: future information may falsify our earlier judgments. Although not explicitly stated, von Wright's account of instrumentality judgments (which have an inherently inductive nature) appear to incorporate the same temporal distinction between (i) collecting past cases and (ii) extending past generalisations to the future in terms of predictions. We believe that the discussed theory of instrumentality can account for the problem of induction through the role of \textit{expectations} in instrumentality judgments:
\begin{quote}
``Judgments of instrumental goodness, usually, even if not necessarily, contain a conjectural element'' \cite[p.~27]{Wri72}.
\end{quote}
In relation to instrumentality, an expectation is a projection of the past onto the nearby future that the action will continue to serve the intended purpose. 
    
 In \cite{Wri57}, von Wright deals with the general problem of induction as an inherently temporal problem that only arises through the involvement of future cases. Likewise, in the setting of practical reasoning, we find specific roles assigned to the past and the future. We believe that expectation as a defeasible cognitive attitude warrants the connection between universal statements about the past and their projection onto the future. Indeed, one can say that even if an agent experienced relations between actions and outcomes in the past, which are not sufficient to form a stable and universally valid judgment about instrumentality, the experience nevertheless serves to support \textit{context-specific and graded} judgments in an agent's practical deliberation. By limiting such judgments to (defeasible) expectations instead of universal claims, the problem of induction is avoided altogether. Despite being defeasible, such judgments may still guide an agent in decision-making. For instance:  
   
   \begin{quote}
    According to what I have experienced so far, using a knife to cut the tape of a parcel
    has served the goal of opening a parcel  
     to a sufficiently good degree, and I expect it to work out in the same way, at least in the near future. 
Using a knife is a good instrument for realising my goal.
   \end{quote}
Thus, the limits of inductive arguments do not prevent one from formulating instrumentality judgments to decide how to act in the immediate future.

Now we know that the tentative definitions (\dfn~\ref{def:candidate_instruments}) provided in \sect~\ref{subsect:philo_good} do not suffice, and reference to the conjectural element must be included. In judging whether an instrument is suitable for a present purpose, we take our 
past
 experience and project it onto the future, conjecturing that its past success will sustain in the immediate future. Hence, we 
recognise
 two \textit{temporal} components in instrumentality judgments: 
(i) past performance of particular actions subsumed under a
 certain
 type and (ii) the expected continuation of the performance of actions of this type in the nearby future.\footnote{We will see that in the formal framework, the expectations of an agent $\alpha$ correspond to
 those (nearby) future moments that
 $\alpha$ regards \textit{likely to happen}.
 We stress that expectations are not to be confused with epistemic notions of (incomplete) knowledge: an agent can have expectations about the future apart from her knowledge of these expected future moments.}
The first temporal component is related to the empirical part of arguments used to establish instrumentality relations; the second temporal component is related to the inductive part of such arguments (see \sect~\ref{Sec:introduction}).

\begin{figure}
    \centering
\begin{tikzpicture}
\node[world] (w0) [label=below:{}] {$w_0$};
\node[world] (w1) [right=of w0, label=below:{$ \phi $}] {$w_1$};
\node[world] (w2) [right=of w1, label=below:{$ \lnot \phi$}] {$w_2$};
\node[world] (w3) [right=of w2, label=below:{$ \phi$}] {$w_3$};
\node[world] (w4) [right=of w3, label=below:{$ \phi $}] {$w_4$};
\node[world] (w5) [right=of w4, xshift=10pt, yshift=45pt, label=right:{$\quad \phi$}] {$w_5$};
\node[world] (w6) [right=of w4, xshift=10pt, label=right:{$\quad \lnot \phi$}] {$w_6$};
\node[world] (w7) [right=of w4, xshift=10pt, yshift=-45pt, label=right:{$\lnot \phi$}] {$\underline{w_7}$};
\node[draw, dashed, color=black, rounded corners, inner xsep=8pt, inner ysep=8pt, fit=(w5)(w6), label={exp-$\alpha$}] (c1) {};
\path[->,draw] (w0) -- (w1) node [midway, yshift=-7pt] {$\Delta$};
\path[->,draw] (w1) -- (w2);
\path[->,draw] (w2) -- (w3);
\path[->,draw] (w3) -- (w4) node [midway, yshift=-7pt] {$\Delta$};
\path[->,draw] (w4) -- (w5) node [midway, yshift=-7pt] {$\Delta$};
\path[->,draw] (w4) -- (w6);
\path[->,draw] (w4) -- (w7) node [midway, yshift=-7pt] {$\Delta$};
\end{tikzpicture}
   
\caption{Expectations and past generalizations of ``$\Delta$ leading to $\phi$'': agent $\alpha$ \textit{expects} (i.e., `exp-$\alpha$' and the dotted lines) the moments $w_5$ and $w_6$ as possible future continuations of $w_4$, although $w_7$ is the actual (underlined) continuation of $w_4$. The arrows labelled with $\Delta$ represent the transitions designated by the performance of $\Delta$.}
   
\label{fig:expectations_generalizations}
\end{figure}
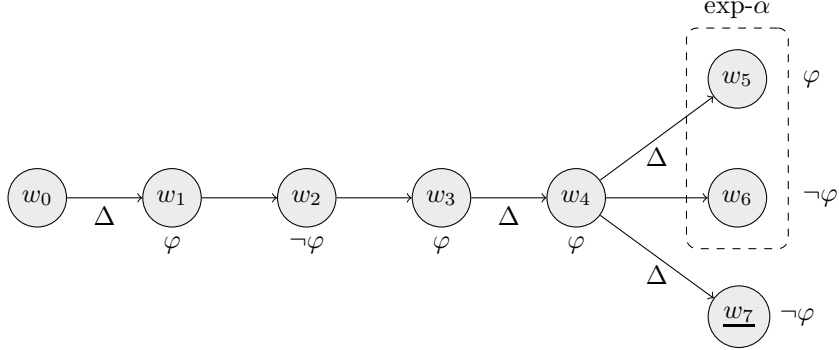

As a basic example of past generalisations and expectations, consider the model presented in \fig~\ref{fig:expectations_generalizations}. Let $\alpha$ be an agent at moment $w_4$. The past (that is, $w_0$ up to $w_4$) provides the agent with the experience that `so far, $\phi$ occurred after every transition caused by $\Delta$'. The agent may generalise this observation to `generally, performing $\Delta$ leads to $\phi$'. Suppose that at $w_4$, the agent believes that `$\Delta$ leading to $\phi$' will continue to hold in the (nearby) future. The moments $w_5$ and $w_6$ express those moments the agent expects to be future continuations of $w_4$. At $w_5$ and $w_6$, it is, in fact, the case that $\Delta$ leads to $\phi$. However, the model shows that the actual future continuation of $w_4$, namely $w_7$, falsifies the generalisation by making $\lnot \phi$ true after the performance of $\Delta$. This captures the idea that future projections are inherently defeasible: an agent's estimations may turn out to be wrong. She may have expected some other future states to be possible or may have wrongly projected her past experience onto the %nearby 
future. The inclusion of expectations thus allows us to adequately address criterion III.

Based on the above, we extend \dfn~\ref{def:candidate_instruments} on candidate instruments to include the conjectural element of instrumentality reasoning. The resulting \dfn~\ref{def:philosophical_instruments} (below) additionally refines our take on criterion I (the first two items) as well as criterion II (the last two items). % thus resulting in the following:

\begin{definition}[Four Definitions of Instrument]\label{def:philosophical_instruments} \

\begin{itemize}
\item[(1)] \textsc{instruments:} An action type $\Delta$ is a $\phi$-instrument for agent $\alpha$ at moment $w$ if and only if (i) $\Delta$ has led to $\phi$ for $\alpha$ at least once in the past of $w$ and (ii) $\alpha$ expects at $w$ that $\Delta$ will lead to $\phi$ in the immediate future.

\item[(2)] \textsc{excellent instruments:} An action type $\Delta$ is an excellent $\phi$-instrument for agent $\alpha$ at moment $w$ if and only if (i) $\Delta$ is a
$\phi$-instrument for $\alpha$ at $w$ and (ii) $\Delta$ has always led to $\phi$ for $\alpha$ in the past of $w$. % and (iii) $\alpha$ expects at $w$ that $\Delta$ will lead to $\phi$ in the immediate future.

\item[(3)] \textsc{better instruments:} An action type $\Delta$ is a \textit{better} $\phi$-instrument for agent $\alpha$ at moment $w$ than the $\phi$-instruments $\Gamma_1,\ldots,\Gamma_n$ available to $\alpha$ at $w$ if and only if (i) $\Delta$ is a  $\phi$-instrument for $\alpha$ at $w$ and (ii) in the past of $w$, action $\Delta$ was more successful for $\alpha$ in guaranteeing $\phi$ 
%has led to $\phi$ more frequently for $\alpha$ 
than $\Gamma_1,\ldots,\Gamma_n$.

\item[(4)] \textsc{good$_n$ instruments:} An action type $\Delta$ is a good$_n$ $\phi$-instrument for agent $\alpha$ at moment $w$ if and only if (i) $\Delta$ is a  $\phi$-instrument for $\alpha$ at $w$ and (ii) $\alpha$'s past performance of $\Delta$ satisfies threshold $n$ at $w$. %, and (iii) $\alpha$ expects at $w$ that $\Delta$ will lead to $\phi$ in the immediate future.

\end{itemize}
\end{definition}
 Since items (2)-(4) in the above definition are based on definition (1), all notions of instrumentality are relative to the agent's expectations about the instrument at the moment of evaluation. In \dfn~\ref{def:philosophical_instruments}, item (3) and (4) refer to judgments of instrumentality involving comparisons, respectively thresholds, as specified in \dfn~\ref{def:comparison_notions}. The logic of actions and expectations introduced in the subsequent sections allows us to capture these definitions formally.

\section{
A Temporal Logic of Actions and Expectations: %Introducing the Formal Setting: The Logic 
$\lae$}\label{Sec:B}\label{Sec:logic_lae}

On the basis of the analysis provided in \sect~\ref{Sec:philosophy_agency} we will develop a logic that will enable us to formalise the mentioned notions of instrumentality. We will refer to the logic as the
% {a} 
 \textit{Temporal Logic of Actions and Expectations}, henceforth $\lae$, since these will be the main ingredients of its syntax and semantics. $\lae$ is a linguistic and deductive extension of $\laei$, a logic developed in \cite{BerPas18}. From the linguistic point of view, the novelty is the use of operators for temporal reference to the past. From the deductive point of view, the novelty is the use of axioms that characterise the behaviour of the new operators, as well as their interaction with the old ones.

For the sake of a self-contained exposition, we will not assume familiarity with $\laei$ and provide a detailed presentation of the new framework, mentioning differences with the old one when needed. We start by listing all fundamental concepts that we intend to capture and then introduce the formal language of $\lae$ in a rigorous way.

\subsection{Fundamental Concepts Expressed}\label{subsect:language}

\noindent
\textit{Purposes}. These are the desired results of actions, represented by formulas $\phi,\psi,\chi,\ldots$ (occasionally annotated). As a matter of fact, a purpose can be equated with a description of a state of affairs. In principle, it is possible to have very complex descriptions involving, for instance, modal concepts (e.g., that it possibly rains or it possibly snows). Furthermore, descriptions of states of affairs may likewise refer to actions (e.g., the description that the door has been opened or the proposition that an agent will open the door).
\\\\
\noindent
\textit{Actions}.
These can be (candidate) instruments for achieving a purpose. We use $\delta_1, \delta_2, \delta_3,\ldots$ to represent atomic action types and build complex action types $\Delta, \Gamma, \Theta, \ldots$ (possibly annotated) via complementation `$-$' (overline), union `$\cup$', and  intersection `$\cap$' of types (e.g., not-opening the door, opening the door or opening the window, and both opening the door and closing the window, respectively). 
\\\\
\noindent
\textit{Agents}. Actions are performed by agents {, which will be here denoted by $\alpha$, $\beta$, $\gamma$, $\ldots$ (possibly annotated)}.
Different actions may be available to different agents and the skills of a particular agent might influence outcomes of actions. We will codify the performance of an action type $\Delta$ by an agent $\alpha$ via a Boolean formula $\phi$. That is to say, 
 {such performance} is associated with a Boolean formula $\phi$ via a function $t$ %for an agent $\alpha$ 
in the following manner: (i) the  {performance of} an atomic action type %s 
$\delta$ %_1,\ldots, \delta_n$ %occurring in $\Delta$ are
 {by $\alpha$ is} translated via $t$ into  {a} \textit{propositional constant} %s 
$\mathfrak{d}^{\alpha}$
%,\ldots, \mathfrak{d_n}^{\alpha}$ %indexed with the name of the intended agent, 
and (ii) the  {performance of a complex action type $\Delta$ by $\alpha$ is translated via $t$ into a formula $\phi$ whose Boolean structure} %logical structure of $\phi$ 
matches the algebraic structure of $\Delta$ %(via the translation function $t$
 {(as} formally defined below). 
Notice the difference between the font of, e.g., $\delta_1$, which indicates an \textit{action type}, and the font of $\mathfrak{d}_1^{\alpha_1}$, which indicates a \textit{proposition}, namely that an action of type $\delta$ is performed by $\alpha_1$. 
\\\\
%\item[1.4] 
\noindent
%\textit{Moments in time}. \\
\textit{Temporal reference and the structure of time.} Von Wright's account of agency takes action as agential change, and change as a transition between moments in time. Consequently, we need to adopt a temporal perspective towards instrumentality as well. In particular, we will take each model for our language to include a set of moments $W$ and evaluate formulae at those moments. 

 {J}udgments of instrumentality refer to the degree in which a certain action, as an instrument, may lead to a certain state of affairs. This `leading to' is the temporal component of instrumentality referring to possible future moments. 
We therefore use a modal operator $\square$, meaning `in all possible immediate next moments'.  {This operator is associated with an accessibility relation between pairs of moments, $R_\square$.}
For instance, let $\mathfrak{d}^{\alpha_1}$ stand for `the door has been opened by agent $\alpha_1$', then the formula $\Box\mathfrak{d}^{\alpha_1}$ is interpreted as `in all possible immediate next moments the door has been opened by agent $\alpha_1$'. We take it that outcomes of actions concern primarily the nearby future, given that an action affects the world as soon as it is performed by an agent. For this reason, in our framework there is no need of using a temporal operator for reference to the whole future of a moment of evaluation (yet, of course, by iterating operator $\square$ it is possible to make reference to a more distant future).

We also use two modal operators referring to the past, $\blacksquare$ and $\hop$, which extend the language of $\laei$. The former is the converse of $\square$ and quantifies over immediate predecessors of a given moment.  The latter is the transitive closure of $\blacksquare$, thus referring to the whole past, not just the immediate past. Hence, $\hop \phi$ reads `everywhere in the past $\phi$ holds'.  {These operators are associated with accessibility relations $R_\blacksquare$ and $R_{\mathsf{H}}$, respectively}. The main benefit of using operators for past reference in $\lae$ is that they allow one to \textit{count} previous performances of an action, keeping track of this counting at the syntactic level. 
 Moreover, counting permits us to compute a success ratio with respect to a certain outcome, giving rise to more refined notions of instrumentality. Such a syntactic procedure of counting was not possible in $\laei$.

Finally, we use an operator $\nbox$
which applies to a formula $\phi$ and says that $\phi$ describes a state of affairs that takes place in the immediate actual future of the moment of evaluation. This allows us to provide a comparison between what the reasoning agent expects to be the case immediately after the moment of evaluation and what will actually be the case. It is associated with an accessibility relation $R_{\nboxs}$ (this operator was originally denoted by $N$ in \cite{BerPas18}).

We stress that there is a twofold asymmetry between the past and the future in our formal framework. First, our indeterministic approach to time allows for several possible immediate successors of a moment, but only one immediate predecessor; i.e., the past of a moment is linear, while its future is possibly branching. Second, the way in which instrumentality judgements are formed, namely through past experience, requires reference to the entire past of a moment, but not to its entire future.
\\\\
\textit{Expectations}.
We employ propositional constants of the form $\expec{i}$ to encode that the most recent expectations of an agent $\alpha_{i}$ are fulfilled. Such formulae further involve the agent's perspective in judging instrumentality relations. Expectations capture the \emph{conjectural element in judgments of instrumentality.}

 \color{black}

\subsection{The Formal Language of $\lae$.}\label{subsect:formal_language}

 Based on the above list, we define two languages: an action language $\langact$, which is an algebra of actions for agent-independent action types, and the logical language $\langlae$ into which these actions will be translated. %an  
We let $\atomacts := \{\delta_{1}, \ldots, \delta_{n}\}$ be a finite set of \emph{atomic action types} and define the set $\actions$ of \emph{all action types} to be all strings generated by the following BNF grammar:
$$
\Delta ::= \delta_{i} \ | \ \Delta \cup \Delta \ | \ \overline{\Delta}
$$
where $\delta_{i} \in \atomacts$. The $\cup$ operation is used to form a \emph{disjunction of action types} (e.g., ‘turning-left or turning-right’) and the --- operation is used to form a \emph{negation of action types} (e.g., ‘not turning-right’). We take the intersection of actions, denoted by the operation $\cap$, as defined, i.e., $\Delta\cap\Gamma:=\overline{\overline{\Delta}\cup\overline{\Gamma}}$.

We use $\agents := \{\alpha_{1}, \ldots, \alpha_{m}\}$ to denote the set of all agent terms and define an \emph{agent-bound action type} to be an expression of the form $\Delta^{\alpha_{i}}$, where $\Delta \in \actions$ and $\alpha_{i} \in \agents$. For any $\alpha_{i} \in \agents$, we let $\wit{\alpha_{i}} := \{ \actag{1}, \ldots, \actag{n}\}$ be the set of propositional constants respectively witnessing the performance of action types $\delta_{1}, \ldots, \delta_{n}$ by $\alpha_{i}$. 
For instance, suppose $\delta_1$ stands for `opening the door', then we read its corresponding witness $\mathfrak{d}_1^{\alpha_1}$ as `the door has been opened by agent $\alpha_1$'. We make the correspondence between agent-bound action types and propositional constants formally precise below. Notice that $|\wit{\alpha_{i}}| = |\actions| = n$. We use $\witset$ to denote the set $\bigcup_{\alpha_{i} \in \agents} \wit{\alpha_{i}}$.  

The language $\lang$ of $\lae$ is defined via the following BNF grammar:
$$
\phi ::= p \ | \ \expec{i} \ | \ \act{i} \ | \ \neg \phi \ | \ \phi \rightarrow \phi \ | \ \wbox \phi \ | \ \nbox \phi \ | \ \bbox \phi \ | \ \h \phi
$$
where $p$ is any \emph{propositional variable} from the set $\var := \{p_{k} \ | \ k \in \mathbb{N} \}$, $i \in \{1, \ldots, m\}$, and $j \in \{1, \ldots, n\}$. We use $p$, $q$, $r$, $\ldots$ (possibly annotated) to denote propositional variables, and $\phi$, $\psi$, $\chi$, $\ldots$ (possibly annotated) to denote formulae from $\lang$. The connectives $\lnot$ and $\rightarrow$ denote `negation' and `material implication', respectively. The interpretation of each modal operator $\wbox$, $\bbox$, $\nbox$, and $\h$ is given in \dfn~\ref{def:semantics} below; we take $\wdia$, $\bdia$, $\ndia$, and $\p$ to be duals of each respective modal operator. Conjunction $\land$, disjunction $\lor$, and material equivalence $\leftrightarrow$ are defined in the usual way. Last, $\expec{i}$ is a propositional constant witnessing the compatibility of a state with $\alpha_{i}$'s expectations, which can be read as `the present state is compatible with $\alpha_i$'s expectations'.

The translation encoding action types from $\langact$ into agent-indexed formulae of $\langlae$ is established recursively through the following function $t$:
\begin{itemize}
    \item[--] For all $\delta\in \atomacts$, and all $\alpha_i\in\agents$,  $t(\delta^{\alpha_i})=\actag{i}$
    \item[--] For all $\Delta\in\langact$, and all $\alpha_i\in\agents$, $t(\overline{\Delta}^{\alpha_i})=\lnot t(\Delta^{\alpha_i})$
    \item[--] For all $\Delta,\Gamma\in\langact$, and all $\alpha_i\in\agents$, $t(\Delta^{\alpha_i} \cup \Gamma^{\alpha_j})=t(\Delta^{\alpha_i})\lor t(\Gamma^{\alpha_j})$
\end{itemize}
The advantage of this translation is that it enables us to reason with actions on the object language level while simultaneously distinguishing such formulae from other (non-action) formulae in the language. This distinction will prove beneficial in (i) defining a variety of modal instrumentality operators in \sect~\ref{Sec:formal_instrumentality} and (ii) axiomatising action specific properties. 
 
To give an example of the expressive power of $\langlae$, we briefly recall the three agentive notions of \textit{would}, \textit{could}, and \textit{will}, as discussed and defined in \cite{BerPas18}.

{\setstretch{\s}\vspace{-0.2cm}
\begin{enumerate}
 [label={},ref=$\mathsf{d\arabic*}$, leftmargin=0.5cm]
   \setcounter{enumi}{0}
\item \label{defq:would} \textit{Would}\\
$[\Delta^{\alpha_i}]^{would}\phi := \square
(t(\Delta^{\alpha_i})\rightarrow\phi)$.
\hfill (\ref{defq:would})
\item \label{defq:could}
\textit{Could}\\
$[\Delta^{\alpha_i}]^{could}\phi := \square
(t(\Delta^{\alpha_i})\rightarrow\phi)\wedge
\lozenge t(\Delta^{\alpha_i})$.
\hfill (\ref{defq:could})
\item \label{defq:will} \textit{Will}\\
$[\Delta^{\alpha_i}]^{will}\phi := \square
(t(\Delta^{\alpha_i})\rightarrow\phi)\wedge \ndia   t(\Delta^{\alpha_i})$.
\hfill (\ref{defq:will})
\end{enumerate}
}

The formula $[\Delta^{\alpha_i}]^{would}\phi$ (\ref{defq:would}) means that `at the current state, by behaving in accordance with $\Delta$, $\alpha_i$ would bring about $\phi$'. The formula $[\Delta^{\alpha_i}]^{could}\phi$ (\ref{defq:could}) means that `at the moment of evaluation, by behaving in accordance with $\Delta$, $\alpha_i$ would bring about $\phi$ and $\alpha_i$ could (i.e., is able to) behave in accordance with $\Delta$'. Finally, the formula $[\Delta^{\alpha_i}]^{will}\phi$ (\ref{defq:will}) means that `at the moment of evaluation, by behaving in accordance with $\Delta$, $\alpha_i$ would bring about $\phi$ and $\alpha_i$ will behave in accordance with $\Delta$'. One can obtain multi-agent variants of the above modalities, such as $[\Delta^{\alpha} \cap \Gamma^{\beta}]^{could}\phi$, referring to the agents $\alpha$ and $\beta$'s ability to jointly secure $\phi$. As an example, let $\Delta$ be the generic action `push' and let $\phi$ stand for `the trolley is rolling'. Then, the formula $[\Delta^{\alpha} \cap \Delta^{\beta}]^{will}\phi$ reads  `$\alpha$ and $\beta$ will both push to ensure the trolley is rolling'.

There are relevant connections between our modalities `could', `would' and `will' and other operators used in the literature on agency logics. For instance, in the tradition of STIT logics the core ingredients are agent-relative operators which connect an agent's (or a group of agents') behaviour to a certain outcome. The most basic of these operators is $[\alpha \; stit]$, which is extensively analysed in \cite{BelPerXu01}. A formula of the form $[\alpha \; stit]\phi$ means that agent $\alpha$ behaves in such a way so as to ensure that $\phi$ holds. Particularly relevant to our setting is the variant of this operator denoted by $[\alpha \; xstit]$ and analysed, e.g., in \cite{Bro11a,Bro11b}, since it involves a temporal shift towards the immediate future:  $[\alpha \; xstit]\phi$ means that agent $\alpha$  behaves in such a way so as to ensure that $\phi$ holds \textit{immediately after}. The role played by these operators in STIT  logics is here captured by the `will' modality. As a matter of fact, the formula $[\Delta^{\alpha_i}]^{will}\phi$ can be regarded as a $\lae$-version of $[\alpha \; xstit]\phi$  which additionally includes information about the action chosen by the agent to get the result. 

There are also significant connections between our approach and proposals to integrate reference to actions in the framework of STIT logics. For instance, the formalism introduced by Xu in \cite{Xu2010} includes formulas of the form $[\alpha,\Delta]\phi$, meaning that agent $\alpha$ obtains $\phi$ by doing $\Delta$. This formalism also includes tense-logical operators for future reference, although not for reference to the immediate future. In our framework the situation is reversed: we have operators that make reference to the immediate future and do not have operators that make reference to the whole future (as explained in Section \ref{subsect:language}).
  
Moreover, as pointed out in \cite{BerPas18},  the `would' modality employed here resembles operators used in  propositional dynamic logic (PDL) 
\cite{FisLad79}. In a broader perspective, our approach can be seen as a reduction of PDL to alethic modal logic with constants witnessing the performance of actions, similar to the  reduction of deontic logic to alethic modal logic proposed by Anderson in
\cite{AndMoo57}.\footnote{One could also add a counterfactual component to the above definitions, namely, the formula $\lozenge \neg \phi$ as a  conjunct. This would strengthen the idea of a causal connection between $\Delta$ and $\phi$ since it may be the case that $\phi$ fails to hold in the immediate future. A counterfactual component is also taken into account in one of  Anderson's strategies to reduce deontic logic, as well as in some STIT logics  (see, e.g., the deliberative STIT operator in ~\cite{BelPerXu01}).}

Finally, we would like to mention that our account of  `would', `could' and `will' is also related to proposals that represent concepts of STIT logic within dynamic logics or logics of temporal computation. For instance, the logic $\mathcal{DLA}$ proposed by Herzig and Lorini  \cite{HerLor10,LorSch17} is based on a language including an agent-indexed operator $[a:  \alpha]$ such that the formula $[a:  \alpha]\phi$ means that agent $a$ ensures that $\phi$ will happen by performing an action of type $\alpha$; this interpretation is very close to our analysis of `will'. Furthermore, Boudou and Lorini in \cite{BouLor18} propose to integrate STIT operators in the  computational logic $\mathsf{CTL}^*$, whose language is endowed with temporal operators for `next-time' ($\mathsf{X}$) and `until' ($\mathsf{U}$). The expressiveness of future reference in the latter setting goes beyond the one allowed for in ours.

\subsection{Semantics and Axioms of $\lae$}
%}
We adopt relational frames for the semantic characterisation of the logic $\lae$. The proposed frame properties will be motivated by the discussion presented in \sect~\ref{Sec:philosophy_agency}. In brief, we will define irreflexive tree-like structures that are linear with respect to the past, and allow for branching with respect to the future. Irreflexivity is motivated by the fact that a moment cannot be its own immediate successor. In other words, the transitive closure of the immediate successor relation is a strict partial order.  One of the advantages of employing irreflexive structures is that it allows for counting with respect to the past, that is, $\bdia$ and $\bdia\bdia$ refer to two distinct moments in the past, the first immediately preceding the moment of evaluation and the second immediately preceding the first. Henceforth, we use $\bdia^i$ ($i\in\mathbb{N}$) to refer to a concatenation of $i$-many $\bdia$ operators, referring to a moment $i$ time units in the past. This feature will be employed in counting successful applications of an instrument serving a certain purpose, thus facilitating the comparative notions of instrumentality from % out in 
\sect~\ref{Sec:philosophy_agency}. %In what follows, 

First, we define \quasiframes \ (\dfn~\ref{def:quasi_frames}), which will subsequently be refined to form the class of envisaged $\lae$-frames (\dfn~\ref{def:lae_frames}). Frames and truth-conditions are defined as usual, with the exception that the interpretation of constants will not be varying per model, but will be fixed on the level of frames (\dfn~\ref{def:semantics}). Namely, the valuation of constants is fixed to sets of moments defined on the frame level which means that the semantic interpretation of such constants is fixed for every model defined over a frame. One advantage of interpreting constants on the level of frames is that we can define frame properties (and corresponding axioms) that will restrict the logical behaviour of certain constants, thus enhancing the expressiveness of the formal setting.

\begin{definition}[\quasiframe, \quasimodel]\label{def:quasi_frames} A \emph{\quasiframe} is a tuple
$$
\mathfrak{F} = \langle W, \{W_{\act{i}} \ | \ \act{i} \in \witset\}, \{W_{\expec{i}} \ | \ \alpha_{i} \in \agents\}, R_\wbox, R_{\nboxs}, R_{\bbox}, R_{\hboxx} \rangle
$$
where $W$ is a set of moments $w$, $u$, $v$, $\ldots$ (which are occasionally annotated), $W_{\act{i}}, W_{\expec{i}} \subseteq W$, and $R_\wbox$, $R_{\nboxs}$, $R_{\bbox}$, and $R_{\hboxx}$ are binary relations over $W$.

A \emph{\quasimodel} $\mathfrak{M} = \langle \mathfrak{F}, V \rangle$ is a tuple %structure
where $\mathfrak{F}$ is a \quasiframe, and $V$ is a valuation function mapping propositional atoms % symbols
to sets of moments such that:

\begin{itemize}

\item[--] $V( \act{i}) := W_{ \act{i}}$

\item[--] $V(\expec{i}) := W_{\expec{i}}$

\end{itemize}

\end{definition}

\begin{definition}[Truth-conditions]\label{def:semantics}
Formulae are evaluated at a state of a model, along the following lines:

\begin{itemize}

\item[--] $\mathfrak{M},w \models p$ \ifandonlyif $w \in V(p)$, for any propositional variable $p \in \var$

\item[--] $\mathfrak{M},w \models \neg \phi$ \ifandonlyif $\mathfrak{M},w \not\models \phi$

\item[--] $\mathfrak{M},w \models \phi \rightarrow \psi$ \ifandonlyif $\mathfrak{M},w \not\models \phi$ or $\mathfrak{M},w \models \psi$

\item[--] $\mathfrak{M},w \models \wbox \phi$ \ifandonlyif for all $v \in W$ s.t. $R_\wbox wv$, it follows that $\mathfrak{M},v \models \phi$

\item[--] $\mathfrak{M},w \models \nbox  \phi$ \ifandonlyif for all $v \in W$ s.t. $R_{\nboxs}wv$, it follows that $\mathfrak{M},v \models \phi$

\item[--] $\mathfrak{M},w \models \bbox \phi$ \ifandonlyif for all $v \in W$ s.t. $R_\bbox wv$, it follows that $\mathfrak{M}, v \models \phi$

\item[--] $\mathfrak{M},w \models \h \phi$ \ifandonlyif for all $v \in W$ s.t. $R_H wv$, it follows that $\mathfrak{M},v \models \phi$

\end{itemize}

The semantic clauses for $\wdia$, $\ndia$, $\bdia$, $\p$, $\land$, $\lor$, and $\leftrightarrow$ are defined as usual.

\end{definition}

For an arbitrary $\Delta^{\alpha_{i}}$ such that $\Delta\in \actions$ and $\alpha_{i} \in \agents$, we define $\w{t(\Delta^{\alpha_{i}})}$ using the following recursive clauses:
\begin{itemize}

\item[--] $\w{t(\delta_j^{\alpha_{i}})} := \w{\mathfrak{d}^{\alpha_{i}}_j}$

\item[--] $\w{t(\overline{\Delta}^{\alpha_{i}})} := W {-} \w{t(\Delta^{\alpha_{i}})}$

\item[--] $\w{t(\Delta^{\alpha_{i}} \cup \Gamma^{\alpha_{j}})} := \w{t(\Delta^{\alpha_{i}})} \cup \w{t(\Gamma^{\alpha_{j}})}$

\end{itemize}
It can be easily shown (through induction on the complexity of $\Delta^{\alpha_i}$) that for each $\Delta\in\langact$, %and
each $\alpha_i\in\agents$, each \quasimodel \ $\mathfrak{M}$ and each moment $w$ in its domain, we have:
$$
\mathfrak{M},w\models t(\Delta^{\alpha_i}) \text{ \ifandonlyif } w\in W_{t(\Delta^{\alpha_i})}
$$
Hence, we obtain the following semantic interpretation of our previously defined operator $[\Delta^{\alpha_i}]^{would}\phi$ (page~\pageref{defq:would}): 
$$
\mathfrak{M},w\models [\Delta^{\alpha_i}]^{would}\phi \text{ \ifandonlyif for all } v\in W \text{ s.t. } R_{\wbox}wv, \text{ if }  v\in W_{ t(\Delta^{\alpha_i})}  \text{ then } \mathfrak{M},v\models\phi
$$
In other words, every immediate successor witnessing the performance of $\Delta$ by an agent $\alpha_{i}$ guarantees the truth of $\phi$. 

\begin{definition}[\laeframe, \laemodel]\label{def:lae_frames} We define a \emph{\laeframe} to be a \quasiframe \ satisfying the following properties:

\vspace{-0.4cm}
\begin{center}
\renewcommand{\arraystretch}{1.4}
\begin{tabular}{l p{9.5cm}}

p(A3)  & For all $w,u,v \in W$, if $R_{\nboxs}wu$ and $R_{\nboxs}wv$, then $w=v$;\footnote{Item $p(A3)$ expresses the functionality of the operator $\nbox$, which is a widespread property among operators used in the literature on agency logic in order to make reference to the actual future. See, e.g., approaches offering a fusion of the next-time operator and STIT operators, such as \cite{HerLor10} and \cite{Bro11a}.}\\

p(A4) & For all $w, u \in W$, if $R_{\nboxs}wu$, then $R_\wbox wu$;\\

p(A5)  & For all $w \in W$ and all distinct agents $\alpha_{1}$, $\ldots$, $\alpha_{n}$, if there are (not necessarily distinct) action types $\Delta_{1}$, $\ldots$, $\Delta_{n}$ s.t. for $1 \leq i \leq n$ there is a $u_{i} \in W$ s.t. $R_\wbox wu_{i}$ and $u_{i} \in \w{t(\Delta_{i}^{\alpha_{i}})}$, there is a $v \in W$ s.t. $R_\wbox wv$ and $v \in \w{t(\Delta_{1}^{\alpha_{1}})} \cap \cdots \cap \w{t(\Delta_{n}^{\alpha_{n}})}$;\\

p(A6)  & For all $w \in W$ and $\alpha_{i} \in \agents$, if there is a $v \in W$ s.t. $R_\wbox wv$ and $v \in \w{\expec{i}}$, then there is also a $u \in W$ s.t. $R_\wbox wu$ and $u \not\in \w{\expec{i}}$.\\

p(A10;A11)  & For all $w,v \in W$, $R_\wbox wv$ iff $R_{\bbox}vw$;\\

p(A12)  & For all $w,u,v \in W$, if $R_{\bbox}wu$ and $R_{\bbox}wv$, then $u = v$;\\

p(A9;A14) & $R_{\hboxx}$ is the transitive closure of   $R_{\bbox}$ (that we will occasionally represent as $R_{\bbox}^{+}$);\\

p(A13)  & For all $w\in W$ either (i) there is no $v$ s.t. $R_{\h} wv$ or (ii) there is some $u$ s.t. $R_{\h} wu$ and there is no $z$ s.t. $R_{\h} uz$. \\ 
\end{tabular}
\end{center}

\noindent A \laemodel \ $\mathfrak{M} = \langle \mathfrak{F},V \rangle$ is a \quasimodel \ such that $\mathfrak{F}$ is a \laeframe. %We let $\frameclass$ denote the set of $\lae$-frames, and $\modelclass$ denote the set of $\lae$-models.
\end{definition}

For the sake of comparability, we have named the frame properties in reference to the axioms of $\lae$ introduced below (see \dfn~\ref{def:axiomatization}). Some of the properties presented above correspond to a set of axioms as opposed to a single axiom; e.g., $p(A9;A14)$ is a property characterised through a combination of the axioms $A9$ and $A14$. 

Let us briefly explain the intuitive meaning of each property: $p(A3)$ ensures that every moment has at most one immediate actual successor. This conveys a simplification of the framework that we exploit to discuss various examples in the rest of the article and that allows us to remain more closely related to the original framework formulated in \cite{BerPas18}. However, reference to $p(A3)$, and the corresponding axiom discussed below, $A3$, can be safely omitted from the technical part of the article (basically, one only has to remove reference to it in the proof of Lemma \ref{lm:canonical-model-is-lae-model}). %moment 
%and 
$p(A4)$ states that this %actual 
successor must be a possible successor. $p(A5)$ expresses the agency property known as \textit{independence of agents}.\footnote{Independence of agents is a fundamental property %of the class of agency logics called \emph{STIT logics}---an acronym standing for `Seeing To It That'. In 
in the setting of STIT logic, where it expresses that any choice available to an agent at a certain moment is compatible with any combination of choices available to the other agents at that moment. In STIT, choices are primitive objects and no object-language reference is made to actions. Here, following \cite{BerPas18}, we adopt independence of agents in a setting in which we have explicit actions available to agents. Such a property ensures that, if an action $\Delta$ is a $\phi$-instrument for $\alpha$, then failing to produce $\phi$ by performing $\Delta$ cannot be caused by the interference of other agents. In other words, $\Delta$ is a proper $\phi$ instrument for $\alpha$. Nevertheless, we will discuss possible failures due to the agent's (false) expectations about the action in question. Future work may be directed to investigating instrumentality in the light of interfering agents, i.e., by dropping independence of agents. See \cite{BelPerXu01} for a discussion.} This property ensures that if an agent can perform a certain action at a certain moment that agent %it
can perform the action at issue irrespective of the actions performed by the other agents. $p(A6)$ ensures that if an agent expects a certain next moment to arise, then there will be another possible next moment that the agent does not expect to arise. $p(A10;A11)$ defines immediate future moments as the inverse of immediate past moments. $p(A12)$ states that the past is linear. $p(A9;A14)$ defines $R_{\h}$ as the transitive closure of $R_{\bbox}$. Last, $p(A13)$ ensures that the past is finite, that is, time has a beginning. This last property will prove useful to comparing candidate instruments serving the same purpose. 

We point out the following fact: 

\begin{theorem}\label{thm:lae_frames_irreflexive}
$\lae$-frames are irreflexive, that is, for all $w\in W$ of a $\lae$-frame $\mathfrak{F}$, we have $(w,w)\not\in R_{\bbox}$, $(w,w)\not\in R_{\wbox}$, $(w,w)\not\in R_{\h}$, and $(w,w)\not\in R_{\nboxs}$.
\end{theorem}
\begin{proof}
Suppose that $R_{\h}ww$. By $p(A13)$, there is some $u\neq w$ s.t. $R_{\h}wu$ and there does not exist a $z$ such that $R_{\h}uz$. By $p(A9;A14)$ there is a sequence of moments $\sigma=v_1,\dots,v_n$ s.t. $v_1=w$, $v_n=u$ and, for $1\leq i \leq n-1$, $R_{\bbox}v_iv_{i+1}$. Furthermore, there is a sequence of moments $\sigma'=v_1',\ldots,v_m'$ s.t. $v_1'=v_m'=w$ and for $1\leq i \leq m-1$, $R_{\bbox}v_i'v_{i+1}'$. 
 We now have three cases to consider due to $p(A12)$ and the fact that $v_1=v_1'$: either (i) $\sigma$ is a proper sub-sequence of $\sigma'$, (ii) $\sigma'$ is a proper sub-sequence of $\sigma$, or (iii) $\sigma = \sigma'$. We show (ii) as (i) and (iii) are simple. Since $\sigma'$ is a sub-sequence of $\sigma$, we know that for some $1 \leq i \leq n-1$, $v_m' = w = v_i$. It follows that $R_{\bbox}wv_{i+1}$. Moreover, we have that $R_{\bbox}wv_{2}'$ by the definition of $\sigma'$, which implies that $v_{i+1} = v_{2}'$ by $p(A12)$. Continuing in this way, one can show that $v_{i+2} = v_{3}'$,  $v_{i+3} = v_{4}'$, etc. It follows that for some $1 \leq j \leq m$, $v_{j}' = v_{n} = u$. However, since $\sigma'$ forms a cycle, this implies that there is some $z$ (namely, $v_{j+1}'$) such that $R_{\bbox}uv_{j+1}$, which further entails that $R_{\h}uv_{j+1}$ by $p(A9;A14)$. This gives a contradiction.
 
 By $p(A9;A14)$ we can infer that for each $w$, $(w,w)\not\in R_{\bbox}$, whence, by $p(A10;A11)$, that $(w,w)\not\in R_{\wbox}$ and finally, by $p(A4)$, that $(w,w)\not\in R_{\nboxs}$.
\end{proof}

\begin{corollary}\label{cor:lae_acyclic}
The relations $\Rwb, \Rbb$, $\Rh$, and $R_{\nboxs}$ of $\lae$-frames are acyclic. 
\end{corollary}

Our axiomatisation for the logic $\lae$ is given below.

\begin{definition}[$\lae$ axiomatisation]\label{def:axiomatization}
The axiomatisation of $\lae$ consists of the following axiom schemes and rules:

\vspace{-0.5cm}
\begin{center}
\renewcommand{\arraystretch}{1.2}
\begin{tabular}{l p{11cm}}
A0 & Any propositional tautology\\
R0 & $\phi, \phi\rightarrow \psi/\psi$\\
A1 & $\wbox (\phi \rightarrow \psi) \rightarrow (\wbox \phi \rightarrow \wbox \psi)$\\
R1 & $\phi/\Box \phi$\\
A2 & $\nbox(\phi\rightarrow \psi) \rightarrow (\nbox \phi\rightarrow \nbox \psi)$\\
A3 & $\ndia \phi \rightarrow \nbox \phi$\\
A4 & $\wbox \phi \rightarrow \nbox \phi$\\
A5 & For any %list of (
distinct $\alpha_1,\dots,\alpha_n \in \agents$ and non-necessarily distinct \\
& $\Delta_1,\dots,\Delta_n \in \actions$, $(\wdia t(\Delta_1^{\alpha_1})\wedge \dots \wedge \wdia t(\Delta_n^{\alpha_n}))\rightarrow \wdia (t(\Delta_1^{\alpha_1})\wedge \dots \wedge t(\Delta_n^{\alpha_n}))$\\

A6 &  For any $\alpha_j \in \agents$, $\wdia \mathfrak{e}^{\alpha_j}\rightarrow \wdia \neg \mathfrak{e}^{\alpha_j}$\\
\end{tabular}
\end{center}
\begin{center}
\renewcommand{\arraystretch}{1.2}
\begin{tabular}{l p{11cm}}
A7 & $\hboxx (\phi \rightarrow \psi) \rightarrow (\hboxx \phi \rightarrow \hboxx \psi)$\\

A8 & $\bbox (\phi \rightarrow \psi) \rightarrow (\bbox \phi \rightarrow \bbox \psi)$\\
A9 & $\h \phi \leftrightarrow (\bbox \phi \land \bbox \h \phi)$\\
A10 & $\phi \rightarrow \wbox \bldia \phi$
\\
A11 & $\phi \rightarrow \bbox \Diamond \phi$
\\
A12 & $\blacklozenge \phi \rightarrow \bbox \phi$
\\
A13 & $\h\bot \vee \hdia \h \bot$
\\
A14 & $\h (\phi \rightarrow \bbox \phi) \rightarrow (\bbox \phi \rightarrow \h \phi)$
\\
R2 & $\phi/\h \phi$\\
\end{tabular}
\end{center}

\noindent
For any formula $\phi \in \lang$, we define $\phi$ to be a \emph{theorem}, and write $\vdash \phi$, \ifandonlyif (i) $\phi$ is an axiom instance, or (ii) $\phi$ is derivable from $\psi$ (and $\chi$) via one of the inference rules where $\vdash \psi$ (and $\vdash \chi$, resp.). We say that $\psi$ is \emph{derivable} from $\Gamma$ in $\lae$, written $\Gamma \vdash \psi$, \ifandonlyif there exist $\phi_{1}, \cdots, \phi_{n} \in \Gamma$ s.t. $\vdash \phi_{1} \land \cdots \land \phi_{n} \rightarrow \psi$.
\end{definition}

 In \dfn~\ref{def:axiomatization} above, we define the notion of a theorem recursively as in Blackburn et al.~\cite[\sect~4.8]{BlaRijVen01}, that is, a theorem is a formula which can be derived via a sequence of axiom instances and rule applications to previoulsy derived theorems. We remark that axioms A1, A2, A7 and A8, together with rules R1, R2, and the derivable rules $A/\nbox A$ and $A/\bbox A$, qualify the modal operators $\wbox$, $\nbox $, $\bbox$, and $\h$ as normal.  Axioms A3 and A12 qualify the accessibility relations associated with $\nbox $ and $\blacksquare$ as functional. Axiom A4 makes the accessibility relation associated with $\square$ a superset of the accessibility relation associated with $\nbox $. Axiom A5 corresponds to the `independence of agents' principle in the STIT-literature~\cite{BelPerXu01}, and states that if each agent can perform a particular action, then all agents can jointly perform such actions. 
 Axioms A9 and A14 are used to express the fact that the accessibility relation associated with $\h$ is the transitive closure of the accessibility relation associated with $\bbox$. Axioms A10 and A11 are used to express the fact that the accessibility relations associated with $\wbox$ and $\bbox$ are reciprocally converse. Finally, axiom A13 says that for any state there is a finite sequence of states related to it via the accessibility relation for $\h$. Taken together, these axioms ensure that the accessibility relations for $\h$, $\wbox$, $\bbox$ and $\nbox$ are irreflexive. However, none of the axioms can ensure this property if taken alone---this follows from general results in correspondence theory for multi-modal, normal logics (cf.~\cite{BlaRijVen01}). We note that the logic $\laei$ originally presented in \cite{BerPas18} is the fragment of $\lae$ without operators $\mathsf{H}$ and $\bbox$ and axiomatized with the deductive principles A0-A6 and R0-R1.

\subsection{Agentive Modals in $\lae$}\label{subsect:volitional_concepts}

In order to get an impression of the expressiveness of $\lae$, we provide some formal definitions of agentive modals in the spirit of von Wright's analysis. Our basic building blocks will be the agentive modals `would' (\ref{defq:would}), `could' (\ref{defq:could}), and `will' (\ref{defq:will}) defined in \sect~\ref{subsect:language}. First, consider the four elementary action types as presented in \fig~\ref{fig:VW_four_types_of_action}: 

{\setstretch{\s}\vspace{-0.2cm}
\begin{enumerate}
 [label={},ref=$\mathsf{d\arabic*}$, leftmargin=0.5cm]
   \setcounter{enumi}{3}
   \item \label{defq:produce}
\textit{Produce}\\
$[\Delta^{\alpha_i}]^{prod} \phi := \lnot \phi \land  [\Delta^{\alpha_i}]^{will} \phi \land \lozenge \lnot \phi$\hfill(\ref{defq:produce})

\item \label{defq:destroy} \textit{Destroy}\\
$[\Delta^{\alpha_i}]^{destr} \phi := \phi \land  [\Delta^{\alpha_i}]^{will} \lnot \phi \land \lozenge \phi$\hfill (\ref{defq:destroy})

\item \label{defq:suppress} \textit{Suppress}\\
$[\Delta^{\alpha_i}]^{supp} \phi := \lnot \phi \land  [\Delta^{\alpha_i}]^{will} \lnot \phi \land \lozenge  \phi$\hfill (\ref{defq:suppress})
\item \label{defq:preserve} \textit{Preserve}\\
$[\Delta^{\alpha_i}]^{pres} \phi := \phi \land  [\Delta^{\alpha_i}]^{will} \phi \land \lozenge \lnot \phi$ \hfill (\ref{defq:preserve})
\end{enumerate}
}

Von Wright's action types are often referred to as \textit{deliberative} in nature; that is, they exclude outcomes which are trivial (e.g., $\top$) and ensure that outcomes are about contingent states of affairs $\phi$, namely, for which $\lozenge\phi$ and $\lozenge \lnot \phi$ hold. As discussed in \sect~\ref{Sec:philosophy_agency}, von Wright's reading of these actions may be too strong, namely, the agent's action decides the faith of $\phi$ completely. For instance, in the case of `producing', through acting the agent ensures $\phi$ whereas through not-acting the agent is able to ensure $\lnot \phi$. In other words, von Wright's account takes the agent's agency as causally sufficient in both directions. In line with our discussion, taking a slightly weaker standpoint (cf.  \textit{causal contribution} in \fig~\ref{fig:indeterministic tree}), we alternatively formalise that the agent has the ability to bring about $\phi$ through performing $\Delta$, but does not bring about $\phi$ by refraining from performing $\Delta$. This is reflected in definitions (\ref{defq:produce}), (\ref{defq:destroy}), (\ref{defq:suppress}), and (\ref{defq:preserve}). (NB. Observe that this position has already been adopted in \cite{Aqv02}.)

By making use of the notions of `would' (\ref{defq:would}) and `could' (\ref{defq:could}), one can provide new versions of the four action types presented above. We call such variations \textit{volitional} concepts. For instance, (\ref{defq:could_destroy}) expresses the idea that an agent $\alpha_i$ could destroy $\phi$ by performing the action $\Delta$. In particular, the first conjunct of (\ref{defq:could_destroy}) states that $\phi$ is presently the case, the second ensures that by performing $\Delta$ the agent $\alpha_i$ would bring about $\lnot \phi$ and $\Delta$ can be performed by $\alpha_i$, and last it is possible that $\phi$ will not be destroyed. 

{\setstretch{\s}
\vspace{-0.2cm}
\begin{enumerate}
 [label={},ref=$\mathsf{d\arabic*}$, leftmargin=0.5cm]
   \setcounter{enumi}{7}
   \item \label{defq:could_destroy} \textit{Could Destroy}\\
$[\Delta^{\alpha_i}]^{could}_{destr} \phi :=  \phi \land [\Delta^{\alpha_i}]^{could} \lnot \phi \land \lozenge \phi$\hfill(\ref{defq:could_destroy})
\end{enumerate}
}

 Last, consider the notion of \textit{forbearance}, which is a stronger agentive notion than merely not acting according to von Wright. To be more precise, forbearing assumes the agent's \textit{ability} to perform the action that is forborne. The formal definition of forbearance (irrespective of its outcome, denoted by `$\top$') is presented in (\ref{defq:forbearance}).

{\setstretch{\s}\vspace{-0.2cm}
\begin{enumerate}
 [label={},ref=$\mathsf{d\arabic*}$, leftmargin=0.5cm]
   \setcounter{enumi}{8}
\item \label{defq:forbearance} \textit{Forbear}\\
 $
[\Delta^{\alpha_i}]^{forb} \top := [\Delta^{\alpha_i}]^{could} \top \land  [\overline{\Delta}^{\alpha_i}]^{will} \top$\hfill(\ref{defq:forbearance})
\end{enumerate}
}

In words, (\ref{defq:forbearance}) reads `the agent $\alpha$ forbears performing action $\Delta$ whenever $\alpha$ could perform action $\Delta$, but will instead perform the action's complement $\overline{\Delta}$'.\footnote{We note that von Wright’s concept of forbearing is different from Belnap et al’s \cite{BelPerXu01} notion of refraining. Briefly, von Wright considers `zero action’ or `passivity’ as a meaningful notion, occurring when an agent does not act and lets the course of nature take over (see page~\pageref{fig:VW_four_types_of_action}). This is captured in the definition of forbearing (\ref{defq:forbearance}), by taking an action corresponding to the negation of all primitive actions. This contrasts with Belnap et al’s account where refraining from acting necessarily corresponds to the agent actively performing some other action.
 } One can see how the notion of forbearance can be extended to incorporate the four elementary action types. The definition provided in (\ref{defq:forbears_produce}) gives an example. 

{\setstretch{\s}\vspace{-0.2cm}
\begin{enumerate}
 [label={},ref=$\mathsf{d\arabic*}$, leftmargin=0.5cm]
   \setcounter{enumi}{9}
\item \label{defq:forbears_produce}\textit{Forbear to Produce}\\
$[\Delta^{\alpha_i}]^{forb}_{prod} \phi := \phi \land [\Delta^{\alpha_i}]^{could}  \phi \land [\overline{\Delta}^{\alpha_i}]^{will} \top \land \lozenge \lnot \phi$\hfill(\ref{defq:forbears_produce})
\end{enumerate}
}

So far, we only considered those temporal operators referring to the future.  We briefly point out that the modularity of combining complex modals of agency extends to reasoning about the past. For instance, one can combine the four elementary action types and the notion of forbearance, with the three notions of `would', `could', and `will', while referring to the agent's past. We investigate reasoning about the past when we formalise  instrumentality in \sect~\ref{Sec:formal_instrumentality}. 

The aim of this section was to demonstrate the high versatility in defining formal notions of agency %in
in the language of $\lae$. One of the reasons for this expressiveness relates to the use of action constants. Namely, the use of constants referring to actions allows us to distinctively reason about actions and states of affairs in a highly modular way, combining them freely with the available temporal operators: future, past, and %\matold{(overall foreseen)}
actual future. In \sect~\ref{Sec:formal_instrumentality}, we demonstrate how this language can be employed to express various instrumentality notions. Furthermore, we will consider agentive modals that arise by involving the notion of  \textit{expectations}. Before moving to our analysis of instrumentality, we demonstrate that $\lae$ is consistent, sound, and weakly complete.

\section{Soundness and Weak Completeness of $\lae$}\label{Sec:weak_completeness}

Due to the interaction between the two new operators for past reference, $\bbox$ and $\mathsf{H}$ (in particular, the fact that we want one to be the transitive closure of the other), proving the completeness of $\lae$ requires a much more complex construction than the one for $\laei$ provided in \cite{BerPas18}. As a matter of fact, the usual canonical  model construction cannot be used  since the logic $\lae$ is not compact (and hence, 
not \emph{strongly} complete, cf.~\cite{BlaRijVen01}): one can prove that the infinite set $\Sigma=\{\bbox^n p: n\in \mathbb{N}\}\cup \{\neg\mathsf{H}p\}$ has no $\lae$-model, whereas each of its finite subsets has some $\lae$-model.
The strategy followed in this section consists of adapting the Fischer-Ladner construction for the completeness of propositional dynamic logic (illustrated in~\cite[Section 4.8]{BlaRijVen01}) in order to obtain a weak completeness result
for our logic $\lae$.  

First, %the logic 
$\lae$ is sound with respect to the class of $\lae$-frames: 
\begin{theorem}[Soundness]
For any formula $\phi\in\langlae$, if $\vdash \phi$, then $\models \phi$. 
\end{theorem}

\begin{proof}
Straightforward by demonstrating that all axioms of $\lae$ are valid for the class of $\lae$-frames and all rules of $\lae$ preserve validity (see \cite{BlaRijVen01}).
\end{proof}

Furthermore, we observe that the logic $\lae$ is consistent. 
\begin{theorem}[Consistency]
The logic $\lae$ is consistent.
\end{theorem}
\begin{proof}
To show $\lae$ consistent, we show that the class of models for $\lae$ is non-empty. We define a \laemodel \ $\mathfrak{M}$ as follows: the set of moments $W := \{w_{i} \ | \ i \in \mathbb{N} \}$, for each $\act{i} \in \witset$, $W_{\act{i}} := \emptyset$, for each $\alpha_{i} \in \agents$, $W_{\expec{i}} := \emptyset$, $R_{\nboxs} := R_{\wbox} := \{(w_{i},w_{i+1}) \ | \ i \in \mathbb{N}\}$, $R_{\bbox}$ is taken to be the converse of $R_{\wbox}$, $R_{\h}$ is the transitive closure of $R_{\bbox}$, and $V$ is taken to be an arbitrary valuation. It is straightforward to verify that $\mathfrak{M}$ is a \laemodel.
\end{proof}

We now define a sequence of concepts that will assist us in establishing our weak completeness result.

\begin{definition}[$\lae$-Closure]\label{def:lae-closure} Let $\Sigma$ be a finite set of formulae. The \emph{$\lae$-closure} of $\Sigma$ is the smallest set $\cl{\Sigma}$ satisfying the conditions below:

\begin{itemize}

\item[--] if $\phi\in \Sigma$ or $\phi$ is a subformula of some $\psi\in \Sigma$, then $\phi\in \cl{\Sigma}$

\item[--] if $\p \phi \in \Sigma$, then $\bdia \p \phi, \bdia \phi \in \cl{\Sigma}$

\item[--] each constant $\mathfrak{d}^{\alpha_j}_i$ and $\mathfrak{e}^{\alpha_j}$ is in $\cl{\Sigma}$

\item[--] if $\wdia t(\Delta_1^{\alpha_1}),\ldots, \wdia t(\Delta_n^{\alpha_n})\in \Sigma$, then $\wdia(t(\Delta_1^{\alpha_1})\wedge\cdots\wedge t(\Delta_n^{\alpha_n}))\in \cl{\Sigma}$

\item[--] if $\wdia \mathfrak{e}^{\alpha_j}\in \Sigma$, then $\wdia\neg \mathfrak{e}^{\alpha_j}\in \cl{\Sigma}$

\item[--] $\h \bot, \p \h \bot \in \cl{\Sigma}$

\end{itemize}

\end{definition}

\begin{definition}[Negation $\sim$]
$$
{\sim} \phi := \begin{cases} 
      \psi & \text{if } \phi \text{ is of the form } \neg \psi, \\
      \neg \phi & \text{otherwise.}
   \end{cases}
$$
For a given finite set of formulae $\Sigma$, we define $\neg \cl{\Sigma}$ to be the smallest extension of $\cl{\Sigma}$ closed under $\sim$ (i.e., under single negations).
\end{definition}

\begin{definition}[Atomic Set]\label{def:atomic-set} Let $\Sigma$ be a finite set of formulae. We say that a set $X$ of formulae is \emph{$\lae$-consistent} \ifandonlyif $X \not\vdash \bot$, and say that a set $X$ of formulae is \emph{maximally $\lae$-consistent} \ifandonlyif $X$ is consistent and for any formula $\phi \not\in X$, $X \cup \{\phi\} \vdash \bot$. A set of formulae $X$ is an \emph{atomic set} over $\Sigma$ \ifandonlyif it is a maximal $\lae$-consistent subset of $\neg \cl{\Sigma}$. We use $\at{\Sigma}$ to denote the set of all atomic sets over $\Sigma$.

\end{definition}

\begin{lemma}\label{lm:atomic-set-properties} Let $\Sigma$ be a finite set of formulae and $X \in \at{\Sigma}$. Then,

\begin{itemize}

\item[(i)] For all $\phi \in \neg \cl{\Sigma}$, either $\phi \in X$ or ${\sim} \phi \in X$, but not both.

\item[(ii)] For all $\phi \in \neg \cl{\Sigma}$, if $X \vdash \phi$, then $\phi \in X$.

\item[(iii)] For all $\phi \lor \psi \in \neg \cl{\Sigma}$, $\phi \lor \psi \in X$ \ifandonlyif either $\phi \in X$ or $\psi \in X$.

\item[(iv)] For all $\p \phi \in \neg \cl{\Sigma}$, $\p \phi \in X$ \ifandonlyif either $\bdia \phi\in X$ or $\bdia \p \phi \in X$.

\end{itemize}
\end{lemma}

\begin{proof} Claims (i)--(iii) are relatively straightforward, so we present the proof of claim (iv) and assume that $\p \phi \in \neg \cl{\Sigma}$.

For the forwards direction, assume that $\p \phi \in X$. By the condition on the $\lae$ closure of a set, $\bdia \p \phi, \bdia \phi \in \neg \cl{\Sigma}$. Since $\h \phi \leftrightarrow (\bbox \phi \wedge \bbox \h \phi)$ is an instance of axiom A9, it follows that $\vdash \p \phi \leftrightarrow (\bdia \p \phi \vee \bdia \phi) $. Then, since $\p \phi \in X$, one can infer that $X\vdash \bdia  \phi \vee \bdia \p \phi$. Suppose that neither $\bdia \phi$ nor $ \bdia \p \phi$ are in $X$, then, due to the definition of $\neg Cl(\Sigma)$ and claim (i), $\neg \bdia \phi,\neg \bdia \p \phi\in X$ and one can infer $X\vdash \neg (\bdia \phi \vee  \bdia \p \phi) $, whence $X\vdash \bot$, which contradicts the fact that $X$ is an atomic set over $\Sigma$.

For the other direction, assume that either $\bdia \p \phi\in X$ or $\bdia \phi \in X$. It follows from axiom A9 that $\vdash (\bdia \p \phi \vee \bdia \phi) \rightarrow \p \phi $, which further implies that $X \vdash \p \phi$, regardless of which case holds. By claim (ii), the assumption that $\p \phi \in \neg \cl{\Sigma}$, and the assumption that $X \in \at{\Sigma}$, we have that $\p \phi \in X$.
\end{proof}

\begin{lemma}
If $\phi \in \neg \cl{\Sigma}$ and $\phi$ is consistent, then there is an $X \in At(\Sigma)$ such that $\phi \in X$.
\end{lemma}

\begin{proof}
Similar to \cite[\lem~4.83]{BlaRijVen01}.
\end{proof}

We note that given a finite set $X$ of formulae, we define $\widehat{X}$ to be a conjunction of all of its elements. Since all such conjunctions are equivalent according to our axiomatisation and semantics, we are free to use any conjunction of the elements of $X$ for $\widehat{X}$.

\begin{definition}[Canonical Model over $\Sigma$]\label{def:CM}
 Let $\Sigma$ be a finite set of formulae. The \emph{canonical model over $\Sigma$} is define to be the tuple
$$
\mathfrak{M}(\Sigma) := \langle W, \{W_{\act{i}} \ | \ \act{i} \in \witset\}, \{W_{\expec{i}} \ | \ \alpha_{i} \in \agents\}, R_\wbox, R_{\nboxs}, R_{\bbox}, R_{\h} \rangle
$$
such that:

\begin{multicols}{2}
\begin{itemize}

\item[--] $W := At(\Sigma)$

\item[--] $Y \in W_{\act{i}}$ \ifandonlyif $\act{i}\in Y$

\item[--] $Y \in W_{\expec{i}}$ \ifandonlyif $\expec{i}\in Y$

\item[--] $YR_{\wbox} Z$ \ifandonlyif $\widehat{Y} \land \Diamond \widehat{Z}$ is consistent

\end{itemize}

\begin{itemize} 

\item[--] $YR_{\nboxs} Z$ \ifandonlyif $\widehat{Y} \land \ndia \widehat{Z}$ is consistent

\item[--] $YR_{\bbox}Z$ \ifandonlyif $\widehat{Y} \land \bdia \widehat{Z}$ is consistent

\item[--] %$YR_{\h}Z$ \ifandonlyif $\widehat{Y} \land \p \widehat{Z}$ is consistent
 $YR_{\h}Z$ \ifandonlyif $YR_{\bbox}^{+}Z$

\item[--] $V(p) := \{Y \in At(\Sigma) \ | \ p \in Y \}$

\end{itemize} 
\end{multicols}

We take $R^+_{\bbox}$ to be the transitive closure of $R_{\bbox}$, and the usual definition of sets of moments associated with complex action types is defined as usual:
\begin{itemize}
\item[--]  $Y\in W_{t(\overline{\Delta}^{\alpha_i})}$ \ifandonlyif $Y\notin W_{t(\Delta^{\alpha_i})}$ 
\item[--] $Y \in \w{t(\Delta^{\alpha_{i}} \cup \Gamma^{\alpha_{j}})}$ \ifandonlyif $Y \in \w{t(\Delta^{\alpha_{i}})} \cup  \w{t(\Gamma^{\alpha_{j}})}$.
\end{itemize}
\end{definition}

\begin{lemma}\label{lm:existence-lemma-1} Let $\Sigma$  be a finite set of formulae and $X \in \at{\Sigma}$.
\begin{itemize}

\item[(i)] If $\qbox \in \{\wbox, \bbox, \nbox\}$ and $\qdia  \in \{\Diamond, \bdia, \ndia\}$, then for all $\ques \phi \in \neg \cl{\Sigma}$, $\ques \phi \in X$ \ifandonlyif there exists a $Y \in At(\Sigma)$ such that $X R_{\qbox} Y$ and $\phi \in Y$.

\item[(ii)] If $\p \phi \in \neg \cl{\Sigma}$, then $\p \phi \in X$ \ifandonlyif there exists a $Y$ such that $X R_{\hboxx} Y$ and $\phi \in Y$.

\end{itemize}
\end{lemma}

\begin{proof} The two claims are proven similar to \lem~4.86 and \lem~4.89 in Blackburn et al.~\cite{BlaRijVen01}, respectively.
\end{proof}

\begin{lemma}\label{lm:canonical-model-is-lae-model}
Let $\Sigma$ be a finite set of formulae. Then, the canonical model $\mathfrak{M}$ over $\Sigma$ is a \laemodel. %, i.e., $\mathfrak{M}(\Sigma) \in \modelclass$.
\end{lemma}

\begin{proof} We know that $\mathfrak{M}(\Sigma)$ is an \quasimodel \ by definition. 
To prove the claim, it suffices to % that $\mathfrak{M}(\Sigma)$ is a $\lae$-model, %$\in \modelclass$, 
%we
argue that $\mathfrak{M}(\Sigma)$ satisfies the properties of a \laemodel. We only show the p(A3) and p(A6) cases as the remaining cases are similar or routine.

\begin{itemize}

\item[p(A3)] Assume that $XR_{\nboxs}Y$ and $XR_{\nboxs}Z$ hold. We want to show that $Y = Z$, that is, we want to show that $\widehat{Y} \land \widehat{Z}$ is consistent (note that since $Y$ and $Z$ are atoms, $\widehat{Y}$ and $\widehat{Z}$ are consistent \ifandonlyif $Y = Z$). We therefore assume that $\widehat{Y} \land \widehat{Z}$ is inconsistent and derive a contradiction. If $\widehat{Y} \land \widehat{Z}$ is inconsistent, then it follows that $\vdash \widehat{Y} \land \widehat{Z} \rightarrow \bot$. By modal reasoning, this implies that $\vdash \ndia (\widehat{Y} \land \widehat{Z}) \rightarrow \ndia \bot$. Observe that $\vdash \ndia \widehat{Y} \land \ndia \widehat{Z} \rightarrow \ndia (\widehat{Y} \land \widehat{Z})$ holds, as it is a consequence of the axiom $\ndia \phi \rightarrow \nbox  \phi$; hence, $\vdash \ndia \widehat{Y} \land \ndia \widehat{Z} \rightarrow \ndia \bot$. By modal and propositional reasoning, we have that $\vdash (\widehat{X} \land \ndia \widehat{Y}) \land (\widehat{X} \land \ndia \widehat{Z}) \rightarrow \bot$, meaning that $\vdash (\widehat{X} \land \ndia \widehat{Y}) \rightarrow \bot \lor (\widehat{X} \land \ndia \widehat{Z}) \rightarrow \bot$. The last theorem implies that either $XR_{\nboxs}Y$ or $XR_{\nboxs}Z$ does not hold (by the definition or $R_{\nboxs}$), thus contradicting our assumption.

\item[p(A6)] Let $X \in W$, $\alpha_{i}$ an agent, and suppose that there exists a $Y$ such that $XR_{\wbox}Y$ and $Y \in \w{\expec{i}}$. We want to show that there exists a $Z \in W$ such that $XR_{\wbox}Z$ and $Z \not\in \w{\expec{i}}$. By \dfn~\ref{def:lae-closure}, we know that $\Diamond \neg \expec{i} \in \neg \cl{\Sigma}$. We aim to show that $\widehat{X} \land \Diamond \neg \expec{i}$ is consistent, since this will imply that $\Diamond \neg \expec{i} \in X$, due to the fact that $X$ is an atomic set. Thus, we assume that $\vdash \widehat{X} \land \Diamond \neg \expec{i} \rightarrow \bot$ to derive a contradiction. By axiom A6, we have  $\vdash \widehat{X} \land \Diamond \expec{i} \rightarrow \bot$ as a consequence, which implies $\vdash \widehat{X} \land \Diamond \expec{i} \land \Diamond \widehat{Y} \rightarrow \bot$ by propositional reasoning. Modal and propositional reasoning may then be applied to derive the following
$$
\vdash \widehat{X} \land \Diamond (\expec{i} \land \widehat{Y}) \rightarrow \bot
$$ 
We know that $\widehat{Y} \land \expec{i}$ is consistent because $Y \in W_{\expec{i}}$. However, since $Y$ is an atomic set and $\expec{i} \in \neg \cl{\Sigma}$, it follows that $\widehat{Y} \land \expec{i}$ is equivalent to $\widehat{Y}$. Hence,
$$
\vdash \widehat{X} \land \Diamond \widehat{Y} \rightarrow \bot
$$ 
contradicting our assumption that $XR_{\wbox}Y$. Consequently, $\Diamond \neg \expec{i} \in X$, so by \lem~\ref{lm:existence-lemma-1}, there exists a $Z \in At(\Sigma)$ such that $XR_{\wbox}Z$ and $\neg \expec{i} \in Z$. The latter fact implies that $Z \not\in \w{\expec{i}}$ by \dfn~\ref{def:CM}.
\qedhere
\end{itemize}
\end{proof}

\begin{lemma}[Truth Lemma]\label{lm:truth-lemma}
Let $\mathfrak{M}(\Sigma)$ be the canonical model over $\Sigma$. For all atomic sets $Y$ and all $\phi \in \neg \cl{\Sigma}$, $\mathfrak{M}(\Sigma), Y \models \phi$ \ifandonlyif $\phi \in Y$.
\end{lemma}

\begin{proof} By induction on the complexity of $\phi$.
\end{proof}

\begin{theorem}[Weak Completeness] For any formula $\phi$, if $\models \phi$, then $\vdash \phi$. %, i.e., $\lae$ is weakly complete.
\end{theorem}

\begin{proof} Follows from \lem~\ref{lm:canonical-model-is-lae-model} and \lem~\ref{lm:truth-lemma}.
\end{proof}

\section{Formal Notions of Instrumentality}\label{Sec:formal_instrumentality}

In this section, we formalize a variety of instrumentality notions corresponding to the philosophical analysis of \sect~\ref{Sec:philosophy_agency}. As will be demonstrated, the logic $\lae$ suffices to capture many of the desired nuances present in judgments of instrumentality. In \sect~\ref{subsect:basic_instruments}, we show how time intervals can be utilised to evaluate an action's suitability for serving a particular purpose (criterion I). In \sect~\ref{good_formal_instruments}, using the various definitions of instrumentality, we semantically define a collection of value judgments qualifying instruments as `better', `best', `worst', `good', and `poor' (criterion II). In \sect~\ref{defeasibility}, we show that, as a result of using past and future references, the obtained formal definitions capture the inherent defeasible nature of judgments of instrumentality (criterion III).

\subsection{Elementary Instrumentality Notions}\label{subsect:basic_instruments}

In \cite{BerPas18}, four formal definitions of instrumentality are provided: a `basic' and a `proper' notion of instrumentality, which are either agent-dependent or agent-independent. (NB. The terms `basic' and `proper' refer to our account of instruments in parts (1) and (2) of \dfn~\ref{def:philosophical_instruments}, respectively, i.e., `plain' and `excellent' instruments.) As remarked above, the logical system $\laei$, introduced in~\cite{BerPas18}, is closely related to the logic $\lae$ presented in this paper. The principal difference between $\laei$ and $\lae$ is that the former's language does not include modalities that refer to the past, whereas the latter does. To be precise, the language of $\lae$ includes the additional $\bbox$ and $\h$ operators. These operators allow us to syntactically define the notion of `\textit{candidate instrument}' and `\textit{excellent candidate instrument}' (see \dfn~\ref{def:candidate_instruments}), referring either to a finite interval of time preceding the moment of evaluation or to the entire past of the moment of evaluation. These notions were merely semantically defined in \cite{BerPas18}. The logical characterisation of the $\bbox$ operator enables us to count successful applications of candidate instruments and syntactically capture intervals of time.

In (\ref{defq:c-instr}) below, we formalise the agent-dependent notion of candidate $\phi$-instrument as presented in item (1) of \dfn~\ref{def:candidate_instruments}. Observe that the definition is relativised to a past time interval with length $n$. 

{\setstretch{\s}
\begin{enumerate}
 [label={},ref=$\mathsf{d\arabic*}$, leftmargin=0.5cm]
   \setcounter{enumi}{10}
   \item \label{defq:c-instr}
\textit{Candidate Instrument for $\alpha$ (with a past interval of length $n$)}\\ 
 $[\Delta^{\alpha}]_{n}^{c-instr}\phi :=  \bigvee\limits_{ 0\leq i\leq n} \bdia^i (t(\Delta^{\alpha}) \land \bdia [\Delta^{\alpha}]^{would}\phi)$
\hfill (\ref{defq:c-instr})
\end{enumerate}
}

The formula $[\Delta^{\alpha}]_{n}^{c-instr}\phi$ (\ref{defq:multi_c-instr}) reads as `somewhere within the past interval of length $n$ there is a moment $v$ at a distance of $i$ units of time that witnessed the successful performance of $\Delta$ by agent $\alpha$ such that at $v$'s immediate predecessor, the performance of $\Delta$ by that agent would have guaranteed $\phi$'.

Agent-dependent excellent candidate $\phi$-instruments are, then, formalized by combining the above definition with the idea that \textit{every} past performance of the relevant action type has led to the intended outcome  (item (2) of \dfn~\ref{def:candidate_instruments}).

{\setstretch{\s}
\begin{enumerate}
 [label={},ref=$\mathsf{d\arabic*}$, leftmargin=0.5cm]
   \setcounter{enumi}{11}
   \item\label{defq:e-c-instr}
\textit{Excellent Candidate Instrument for $\alpha$ (with a past interval of length $n$)}\\ 
 $[\Delta^{\alpha}]_{n}^{ex\matnew{.}c-instr}\phi := [\Delta^{\alpha}]_{n}^{c-instr}\phi  \wedge \bigwedge\limits_{ 1\leq k\leq n} \bbox^k  [\Delta^{\alpha}]^{would}\phi$ 
 \hfill(\ref{defq:e-c-instr})
 \end{enumerate}}

We read (\ref{defq:e-c-instr}) as `action type $\Delta$ has proven to be a candidate instrument for $\phi$ for $\alpha$ at least once in the interval, and every other performance of $\Delta$ by $\alpha$ within the interval would have also guaranteed $\phi$'. One can say that,  within the past interval $n$, the action type $\Delta$ has a one hundred percent success rate for $\alpha$ in obtaining $\phi$. In the sequel, we introduce ways of refining these definitions.\footnote{Observe that in  (\ref{defq:c-instr}) and (\ref{defq:e-c-instr}) reference to past experience---i.e., $\bdia^i$ and $\bbox^i$---also includes reference to outcomes obtained at the moment of evaluation $w$. %Hence, $i$ may equal $1$. 
Otherwise, a performance of $\Delta$ immediately before $w$ might fail to deliver $\phi$ at $w$.} 

Agent-independent generalizations of (\ref{defq:c-instr}) and (\ref{defq:e-c-instr}) are captured by (\ref{defq:multi_c-instr}), respectively (\ref{defq:generalized_e-c-instr}) (cf. the agent-dependent definitions of \cite{BerPas18}).

{\setstretch{\s}
\begin{enumerate}
 [label={},ref=$\mathsf{d\arabic*}$, leftmargin=0.5cm]
   \setcounter{enumi}{12}
   \item \label{defq:multi_c-instr}
   $[\Delta]_{n}^{c-instr}\phi := \bigvee\limits_{\alpha\in \agents} \bigvee\limits_{ 0\leq i\leq n} \bdia^i (t(\Delta^{\alpha}) \land \bdia[\Delta^{\alpha}]^{would}\phi)$
\hfill (\ref{defq:multi_c-instr})
\end{enumerate}
}

{\setstretch{\s}
\begin{enumerate}
 [label={},ref=$\mathsf{d\arabic*}$, leftmargin=0.5cm]
   \setcounter{enumi}{13}
 \item \label{defq:generalized_e-c-instr}
 $[\Delta]_{n}^{ex\matnew{.}c-instr}\phi := [\Delta]_{n}^{c-instr}\phi  \wedge \bigwedge\limits_{1\leq k\leq n} \bbox^k \bigwedge\limits_{\alpha\in \agents} [\Delta^{\alpha}]^{would}\phi$ 
 \hfill(\ref{defq:generalized_e-c-instr})
 \end{enumerate}
 }

Henceforth, we focus on \textit{agent-dependent} notions. The agent-independent versions of each of the definitions below can be straightforwardly obtained.

The more refined definitions of instruments and excellent instruments---corresponding to items (1) and (2) of \dfn~\ref{def:philosophical_instruments}---are obtained by adding the pivotal conjectural element, reflecting the agent's expectations about the instrument's suitability in the immediate future. To capture these notions, we first need to alter the agentive operator \textit{would} (\ref{defq:would}) (\sect~\ref{subsect:formal_language}).  %and \textit{could} (\ref{defq:could}) introduced in \cite{BerPas18}.

{\setstretch{\s}
\begin{enumerate}
 [label={},ref=$\mathsf{d\arabic*}$, leftmargin=0.5cm]
   \setcounter{enumi}{14}
\item\label{defq:exp_would} \textit{Expected Would}\\
%For any $\Delta\in \langact$ and $\alpha\in \agents$,\\
$[\Delta^{\alpha}]^{would}_{ex}\phi:=\square
((t(\Delta^{\alpha})\land \mathfrak{e}^{\alpha})\rightarrow\phi)$.
\hfill(\ref{defq:exp_would})
\end{enumerate}
}

The `expected would' operator (\ref{defq:exp_would}) %and (\ref{defq:exp_could}) 
%of `would' %and `could' 
restricts the formula's evaluation to immediate future moments that %belong to those moments 
the agent \textit{expects} as %possible future 
continuations of the present. Using (\ref{defq:exp_would}), we formalize items (1) and (2) of \dfn~\ref{def:philosophical_instruments} as follows:
%We note that $\alpha_j\in\agents$ in the definitions below.

{\setstretch{\s}
\begin{enumerate}
 [label={},ref=$\mathsf{d\arabic*}$, leftmargin=0.5cm]
   \setcounter{enumi}{15}
   \item\label{defq:f-instr}
\textit{Instrument for $\alpha$ (with a past interval of length $n$)}\\
 $[\Delta^{\alpha}]_{n}^{instr}\phi :=  \bigvee\limits_{ 0\leq i\leq n} \bdia^i (t(\Delta^{\alpha}) \land \bdia [\Delta^{\alpha}]^{would}\phi) \land [\Delta^{\alpha}]^{would}_{ex}\phi$
%\\\\
\hfill(\ref{defq:f-instr})

\item \label{defq:e-f-instr}
\textit{Excellent Instrument for $\alpha$ (with a past interval of length $n$)}\\
 $[\Delta^{\alpha}]_{n}^{ex-instr}\phi := [\Delta^{\alpha}]_{n}^{instr}\phi  \wedge \bigwedge\limits_{1\leq k\leq n} \bbox^k  [\Delta^{\alpha}]^{would}\phi$ 
 \hfill(\ref{defq:e-f-instr})%\\\\
\end{enumerate}
}

The final formalisations (\ref{defq:f-instr}) and (\ref{defq:e-f-instr})---to which we sometimes refer as `proper instruments'---differ from their candidate counterparts (\ref{defq:c-instr}) and (\ref{defq:e-c-instr}) through the additional conjunct expressing that the agent expects that, at the %present
moment of evaluation, she would %could 
guarantee $\phi$ by performing $\Delta$. Stronger notions of (excellent) instruments are straightforwardly obtained by using a definition of `expected could' instead of `expected would' as the last conjunct in  (\ref{defq:f-instr}).\footnote{Such stronger notions are used in \cite{BerPas18}. As pointed out by a reviewer, % for the present article,
the use of expected could can be problematic. For instance, suppose that at $w$ agent $\alpha$ acknowledges a relation between past performances of action $\Delta$ (e.g., turning a car's ignition key) and an outcome $\phi$ (the car's motor is running). Now, suppose that at $w$, $\alpha$ expects that $\Delta$ will serve purpose $\phi$ in the immediate future. This may suffice for $\alpha$ to consider $\Delta$ an instrument for $\phi$ at $w$, even if $\Delta$ cannot be performed at $w$ for reasons unknown to $\alpha$ (e.g., the car's battery is down).}
The following implications are theorems of $\lae$ and show the various relations between the four notions defined so far.

{\setstretch{\s}
\begin{enumerate}
 % [label={},ref=$\mathsf{d\arabic*}$, leftmargin=0.5cm]
   % \setcounter{enumi}{18}
    \item[] %\label{defq:consequences1}
$[\Delta^{\alpha}]_{n}^{ex-instr}\phi\rightarrow [\Delta^{\alpha}]_{n}^{ex\matnew{.}c-instr}\phi\ $ and %(ii)
$\ [\Delta^{\alpha}]_{n}^{instr}\phi\rightarrow [\Delta^{\alpha}]_{n}^{c-instr}\phi$%\hfill (\ref{defq:consequences1})

\item[] %\label{defq:consequences2}
$[\Delta^{\alpha}]_{n}^{ex\matnew{.}c-instr}\phi\rightarrow [\Delta^{\alpha}]_{n}^{c-instr}\phi\ $ and %(ii)
$\ [\Delta^{\alpha}]_{n}^{ex-instr}\phi\rightarrow [\Delta^{\alpha}]_{n}^{instr}\phi$%\hfill(\ref{defq:consequences2})
\end{enumerate}
}

Before moving to comparative notions of instrumentality, we make three remarks. First, purposes may include action formulae or action witnesses. For instance, agent $\alpha$'s purpose may be to ensure that agent $\beta$ could bring about $\phi$ by performing $\Delta$. In that case, $\alpha$ is looking for the action that will ensure that, at the next moment, $[\Delta^{\beta}]\phi$ holds. We will not further pursue this here.

Second, we did not consider deliberative versions of instrumentality. Definitions (\ref{defq:c-instr}), (\ref{defq:e-c-instr}), (\ref{defq:f-instr}), and (\ref{defq:e-f-instr}) will qualify each action type as an instrument for bringing about %tautologies. 
\textit{tautologous} propositions. Deliberative variants can be straightforwardly obtained in the spirit of \cite{BelPerXu01}'s deliberative STIT operator.

Third, the volitional concepts of \sect~\ref{subsect:volitional_concepts} (e.g., `producing' and `destroying') can also be employed in the context of instrumentality. Such definitions can be straightforwardly given in the framework of $\lae$. For instance, (\ref{defq:produce_instr}) (below) expresses that, for agent $\alpha$, $\Delta$ is an instrument for \textit{producing} $\phi$ because (i) $\alpha$ produced $\phi$ through performing $\Delta$ at least once in the past (with an interval of length $n$) and (ii) $\alpha$ \textit{expects} to produce $\phi$ through $\Delta$ at present. 

{\setstretch{\s}
\begin{enumerate}
 [label={},ref=$\mathsf{d\arabic*}$, leftmargin=0.5cm]
   \setcounter{enumi}{17}
   \item\label{defq:produce_instr}
\begin{tabular}{l c l}
$[\Delta^{\alpha}]_{prod-n}^{instr}\phi$ & $:=$ & $(i) \  \bigvee\limits_{ 0\leq i\leq n} \bdia^i (t(\Delta^{\alpha}) \land \bdia [\Delta^{\alpha}]^{would}_{prod}\phi) \ \land$ \\
& &  $(ii) \quad [\Delta^{\alpha}]^{would}_{ex}\phi \land \lnot\phi \land \lozenge (\lnot \phi\land \mathfrak{e}^{\alpha})$\\     
\end{tabular}\hfill (\ref{defq:produce_instr})
\end{enumerate}}
\noindent That (\ref{defq:produce_instr}) is a stronger notion than (\ref{defq:f-instr}) follows from the following theorem: %

\begin{enumerate}
\item[] $[\Delta^{\alpha}]_{prod-n}^{instr}\phi \rightarrow [\Delta^{\alpha}]_{n}^{instr}\phi\ $
\end{enumerate}% }{defq:produce_instr_implies_f-instr}

\subsection{Good Instruments and Comparative Instrumentality}\label{good_formal_instruments}

As shown in the previous section, our language $\langlae$ is suitable for the definition of candidate instruments, as per (\ref{defq:c-instr}) and (\ref{defq:e-c-instr}), as well as proper instruments, as per (\ref{defq:f-instr}) and (\ref{defq:e-f-instr}). We now discuss how to evaluate different instruments serving the same purpose, giving rise to notions such as `better' and `good' instruments (see \dfn~\ref{def:comparison_notions} and items (3) and (4) of \dfn~\ref{def:philosophical_instruments}). In our framework, we can use $\lae$-models to determine whether an available (candidate) instrument is regarded as better than another. In this section, we proceed with a semantic analysis of comparative instrumentality, providing various notions of \textit{better (worse), best (worst),} and \textit{good (poor)} instruments. 
 
We adopt the following methodology: % for axiologically qualifying $\phi$-instruments:
\begin{enumerate}
\item Collect all the instruments that serve $\phi$;
\item Determine the success ratio of the obtained instruments w.r.t. $\phi$.
%\footnote{Recall, an action may, as an instrument, serve several different purposes.}
\end{enumerate}
For comparative judgments of \textit{betterness}, we proceed accordingly:
\begin{enumerate}
\setcounter{enumi}{2}
\item%[$(\mathsf{s3})$]
Order the available instruments on the basis of their success ratio; 
\item%[$(\mathsf{s4})$]
Identify the best and the worst;
\item%[$(\mathsf{s5})$]
Identify better instruments by comparing instruments within their order.
\end{enumerate}
For comparative judgments of instrumental \textit{goodness}, we proceed as follows: 
\begin{enumerate}
\setcounter{enumi}{5}
\item%[$(\mathsf{s6})$]
Identify good instruments by checking whether their success ratio satisfies a certain threshold ratio;
\item%[$(\mathsf{s7})$]
Identify good instruments by checking whether they satisfy certain additional thresholds, such as a minimum amount of past witnesses. 
\end{enumerate}
We addressed step (1) in the previous section via (\ref{defq:f-instr}). In \sect~\ref{subsect:better}, we deal with (3)-(5), and in \sect~\ref{subsec:good}, we deal with (6) and (7). First, let us address the notion of \textit{success ratio} from step (2). For readability, in what follows, we omit explicit reference to a past interval of length $n$ without a loss of generality. Recall that $\lae$-frames are rooted tree-like structures, meaning each moment $w$ has a finite past. Thus, there are only finitely many witnesses of an instrument's application. We use `$\infty$' to denote consideration of the \textit{entire past}. That is, for a moment $w\in W$,  $w\models [\Delta^{\alpha}]^{instr}_{\infty}\phi$ signifies that we evaluate the past of $w$ up to the root of the model.

In the sequel, we provide our definitions for proper instruments only (\ref{defq:f-instr}) (variations can be straightforwardly obtained). Let $\phi$ be the purpose endorsed by agent $\alpha$ at moment $w$. Then, we let $\instr{w}{\alpha}{\phi}$ denote the set containing all $\phi$-instruments available to $\alpha$ at $w$ with respect to the entire past of $w$, i.e., all actions that served $\phi$ at least once, and of which $\alpha$ expects that it will serve $\phi$ again.\footnote{Note that even if $\instr{w}{\alpha}{\phi}$ contains infinitely many action types, up to logical equivalence, the set will only contain finitely many. We come back to this in \sect~\ref{subsect:better}.}  Definition (\ref{defq:phi_instruments}) below makes this formally precise. 

{\setstretch{\s}
\begin{enumerate}
 [label={},ref=$\mathsf{d\arabic*}$, leftmargin=0.5cm]
   \setcounter{enumi}{18}
   \item\label{defq:phi_instruments}
\textit{Available $\phi$-instruments for $\alpha\in\agents$ at $w$}\\
$\instr{w}{\alpha}{\phi} := \{ \Delta\in \actions \ | \ w\models [\Delta^{\alpha}]^{instr}_{\infty}\phi \}$\hfill(\ref{defq:phi_instruments})
\end{enumerate}}

For each collected available instrument $\Delta\in\instr{w}{\alpha}{\phi}$, we must determine how `successful' it is at securing $\phi$. Here, we opt for defining success in terms of an achievement/failure ratio. Below, \textit{achievement} (\ref{defq:achievement}) refers to the fact that by performing $\Delta$ agent $\alpha$ would secure $\phi$ and, in fact, did secure $\phi$.

{\setstretch{\s}
\begin{enumerate}
 [label={},ref=$\mathsf{d\arabic*}$, leftmargin=0.5cm]
   \setcounter{enumi}{19}
   \item\label{defq:achievement}
   \textit{Achievement of $\Delta\in \instr{w}{\alpha}{\phi}$}\\
    $\achiev{w}{\alpha}{\phi}{\Delta} := \{ \theta\ |\ \theta= \bdia^i(t(\Delta^{\alpha})\land\bdia[\Delta^{\alpha}]^{would}\phi), 0 {\leq} i$, and $w\models \theta\}$
\hfill(\ref{defq:achievement})
\end{enumerate}}

In (\ref{defq:achievement}), we use $0\leq i$ to indicate that the %present 
moment of evaluation may serve as a witness of a successful application of $\Delta$ initiated in its immediate preceding moment. Furthermore, the first conjunct in $\theta$ witnesses the  {actual} 
%\matold{overall foreseen} 
performance of $\Delta$ and attainment of $\phi$, whereas the second conjunct ensures that the simultaneous occurrence of $\Delta$ and $\phi$ is not merely coincidental. 

We define \textit{failure} (\ref{defq:failure}) in terms of the expectations of the agent involved. % in using the instrument.
A failure to secure $\phi$ by $\alpha$'s performance of $\Delta$ at $w$ means that (i) at $w$, $\Delta$ has just been performed by $\alpha$, (ii) $\phi$ does not hold at $w$, and (iii) at the immediately preceding moment  $\alpha$ expected that by performing $\Delta$, $\phi$ would be attained. 

{\setstretch{\s}
\begin{enumerate}
 [label={},ref=$\mathsf{d\arabic*}$, leftmargin=0.5cm]
   \setcounter{enumi}{20}
   \item\label{defq:failure}
   \textit{Failure of $\Delta\in \instr{w}{\alpha}{\phi}$}\\
$\fail{w}{\alpha}{\phi}{\Delta} {:=} \{  \theta\  |\ \! \theta=\bdia^i(t(\Delta^{\alpha})\land\!\lnot\phi \land \bdia[\Delta^{\alpha}]^{would}_{ex}\phi)$, $0{\leq}i$, and $w\models \theta\}$\hfill(\ref{defq:failure})
\end{enumerate}}

Through the use of $\bdia^i$, all members of $\achiev{w}{\alpha}{\phi}{\Delta}$ and $\fail{w}{\alpha}{\phi}{\Delta}$ represent unique past moments. In what follows, we sometimes abuse notation and let these sets  stand for the sets of moments witnessing the respective achievements and failures. Furthermore, the  two sets are finite since the past is finite, and so,  %paths are finite.
the cardinality of the two sets in (\ref{defq:achievement}) and (\ref{defq:failure}) can be utilised to count the amount of achieved or failed applications of the instrument under consideration. Thus, we have the means to calculate a \textit{success ratio} for the collected instruments, i.e., (\ref{defq:success_ratio}). We denote the cardinality of a set $S$ by $|S|$. 

{\setstretch{1}
\begin{enumerate}
 [label={},ref=$\mathsf{d\arabic*}$, leftmargin=0.5cm]
   \setcounter{enumi}{21}
   \item\label{defq:success_ratio}
   \textit{Success ratio of $\Delta\in \instr{w}{\alpha}{\phi}$}
    \item[]
$ \success{w}{\alpha}{\phi}{\Delta} :=\displaystyle \frac{|\achiev{w}{\alpha}{\phi}{\Delta}| }{|\achiev{w}{\alpha}{\phi}{\Delta} \cup \fail{w}{\alpha}{\phi}{\Delta}|}$
\hfill(\ref{defq:success_ratio})
\end{enumerate}
}

To exemplify the above machinery, consider \fig~\ref{fig:success_ratio}  and suppose there is an agent $\alpha$ at moment $w_6$ who aims to bring about $\phi$. There are two available instruments at $w_6$, namely, $\Delta$ and $\Gamma$, whose potential was witnessed at $w_2$. We find that $\alpha$ expects both instruments to serve $\phi$ at $w_6$ (see $w_8$), although $\alpha$'s current expectations about $\Gamma$ are incorrect (see $w_9$). Furthermore, notice that, from the vantage point of $w_6$, $\Gamma$ and $\Delta$ both have two achievement witnesses, namely, $\achiev{w_6}{\alpha}{\phi}{\Gamma}=\{ w_2, w_5\}$, and $\achiev{w_6}{\alpha}{\phi}{\Delta}=\{ w_1, w_2\}$. Furthermore, $\Delta$ also has a witness of a failed application at $w_6$, that is, $\fail{w_6}{\alpha}{\phi}{\Gamma}=\{ w_6\}$. At $w_5$, $\alpha$ had expectations that turned out to be wrong at $w_6$. Hence, the success ratios of $\Gamma$ and $\Delta$ at $w_6$ are $\success{w_6}{\alpha}{\phi}{\Gamma}=1$ and $\success{w_6}{\alpha}{\phi}{\Delta}=\frac{2}{3}$, respectively.

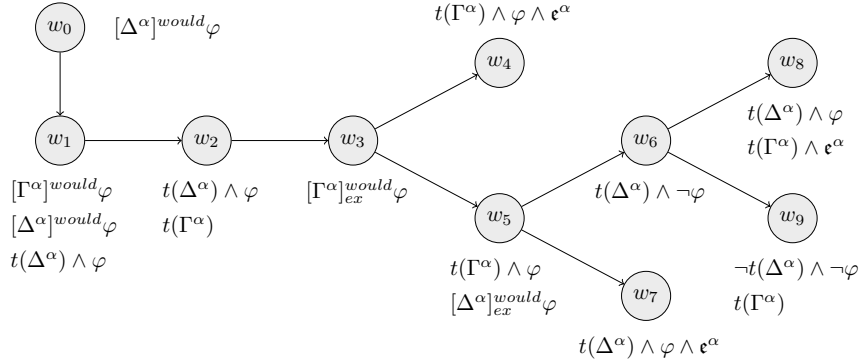
\begin{figure}
    \centering
    \resizebox{\columnwidth}{!}{
    \begin{tikzpicture}
    \node[world] (w0) [label=below:{\def\arraystretch{1.3}
    \begin{tabular}{l}
    $[\Gamma^{\alpha}]^{would}\phi$\\
    $[\Delta^{\alpha}]^{would}\phi$\\
    $t(\Delta^{\alpha})\land \phi$
    \end{tabular}}] {$w_1$};
    \node[world] (wx) [above=of w0, label=right:{\def\arraystretch{1.3}
    \begin{tabular}{l}
    $[\Delta^{\alpha}]^{would}\phi$\\
    %$t(\Gamma^{\alpha})$\\
    \end{tabular}}] {$w_0$};
    \node[world] (w1) [right=of w0, xshift=15pt, label=below:{\def\arraystretch{1.3}
    \begin{tabular}{l}
    $t(\Delta^{\alpha})\land\phi$\\
    $t(\Gamma^{\alpha})$\\
    \end{tabular}}] {$w_2$};
    \node[world] (w2) [right=of w1, xshift=15pt, 
    label=below:{\def\arraystretch{1.3}
    \begin{tabular}{l}
    $[\Gamma^{\alpha}]_{ex}^{would}\phi$\\
    \end{tabular}}] {$w_3$};
    \node[world] (w3) [right=of w2, xshift=15pt, yshift=35pt, label=above:{\def\arraystretch{1.3}
    \begin{tabular}{l}
    $t(\Gamma^{\alpha})\land\phi \land \mathfrak{e}^{\alpha}$\\
    \end{tabular}}] {$w_4$};
    \node[world] (w4) [right=of w2, xshift=15pt, yshift=-35pt, label=below:{\def\arraystretch{1.3}
    \begin{tabular}{l}
    $t(\Gamma^{\alpha})\land \phi$\\
    $[\Delta^{\alpha}]_{ex}^{would}\phi$
    \end{tabular}}] {$w_5$};
    \node[world] (w5) [right=of w4, xshift=15pt, yshift=35pt, label=below:{\def\arraystretch{1.3}
    \begin{tabular}{l}
    $t(\Delta^{\alpha}) \land \lnot \phi$
    \end{tabular}}] {$w_6$};
    \node[world] (w6) [right=of w4, xshift=15pt, yshift=-35pt, label=below:{\def\arraystretch{1.3}
    \begin{tabular}{l}
    $t(\Delta^{\alpha}) \land \phi \land \mathfrak{e}^{\alpha}$
    \end{tabular}}] {$w_7$};
    \node[world] (w7) [right=of w5, xshift=15pt, yshift=35pt, label=below:{\def\arraystretch{1.3}
    \begin{tabular}{l}
    $t(\Delta^{\alpha}) \land \phi$\\
    $t(\Gamma^{\alpha}) \land \mathfrak{e}^{\alpha}$
    \end{tabular}}] {$w_8$};
    \node[world] (w8) [right=of w5, xshift=15pt, yshift=-35pt, label=below:{\def\arraystretch{1.3}
    \begin{tabular}{l}
    $\lnot t(\Delta^{\alpha}) \land \lnot \phi$\\
    $t(\Gamma^{\alpha})$
    \end{tabular}}] {$w_9$};
    \path[->,draw] (w0) -- (w1);
    \path[->,draw] (w1) -- (w2);
    \path[->,draw] (w2) -- (w3);
    \path[->,draw] (w2) -- (w4);
    \path[->,draw] (w4) -- (w5);
    \path[->,draw] (w4) -- (w6);
    \path[->,draw] (w5) -- (w7);
    \path[->,draw] (w5) -- (w8);
    \path[->,draw] (wx) -- (w0);
    \end{tikzpicture}
    }
    \caption{An example of evaluating $\phi$-instruments $\Gamma$ and $\Delta$ for $\alpha$ at $w_6$. %\kees{I added $w_x$ later to accommodate the example to another point I want to make.}
    }
    \label{fig:success_ratio}
\end{figure}

\subsubsection{Comparative Judgments of Betterness}\label{subsect:better}

The notion of success ratio is pivotal for defining various types of axiological---i.e., value---judgment concerning instrumentality. In particular, we are interested in the following axiological concepts: \textit{best}, \textit{worst}, \textit{better}, \textit{good}, and  \textit{poor}. The latter two are formally addressed in the next section. 

Recall that the language $\mathcal{L}_{\lae}$ allows only for 
finitely many action types, up to provable equivalence of the formulas that witness their performance by an agent. That is, since $\actions$ is constructed over a finite number of atomic action types in $\atomacts$, for each agent in $\agents$, there will be finitely many equivalence classes $[\![ \Delta^{\alpha} ]\!]~:=~\{\Gamma^{\alpha} \ | \ \vdash_{\lae} t(\Delta^{\alpha}) \leftrightarrow t(\Gamma^{\alpha}) \}$ of equivalent actions. We let $ \mathsf{EqAct} =\{ [\![ \Delta^{\alpha} ]\!]\ |\ $ for $\Delta\in\actions$ and $\alpha\in\agents \}$ be the set of all such equivalence classes.
Consequently, we obtain a finite ordering of actions (up to equivalence) when ordering candidate instruments. This observation enables us to identify those instruments at the ordering's upper- and lower-bound.
\\
\\
\textbf{Naive best, worst, and better.} We define $\better{w}{\phi}$ to be a success ratio ordering over actions such that, for each $\Delta,\Gamma\in\instr{w}{\alpha}{\phi}$, we have

{\setstretch{1}
\begin{enumerate}
 [label={},ref=$\mathsf{d\arabic*}$, leftmargin=0.5cm]
   \setcounter{enumi}{22}
   \item\label{defq:better}
$ \Delta^{\alpha} \better{w}{\phi} \Gamma^{\alpha} \text{ \ifandonlyif } \success{w}{\alpha}{\phi}{\Delta} \geq \success{w}{\alpha}{\phi}{\Gamma}.$\hfill(\ref{defq:better})
\end{enumerate}}

We define $\Delta^{\alpha} \sbetter{w}{\phi} \Gamma^{\alpha}$ as the conjunction $\Delta^{\alpha} \better{w}{\phi} \Gamma^{\alpha}$ and $\Gamma^{\alpha} \not\better{w}{\phi} \Delta^{\alpha}$. We interpret $\better{w}{\phi}$ as a ``betterness'' relation: i.e., we read $\Delta^{\alpha} \better{w}{\phi} \Gamma^{\alpha}$ as `at $w$, $\Delta$ is a \textit{weakly better} instrument for agent $\alpha$ to secure $\phi$ than the instrument $\Gamma$' (i.e., $\Delta$ is at least as good as $\Gamma$ for $\alpha$ in order to get $\phi$). Then, $\Delta^{\alpha} \sbetter{w}{\phi} \Gamma^{\alpha}$ expresses that $\Delta$ is a (strictly) \textit{better} $\phi$-instrument than $\Gamma$, for $\alpha$ at $w$.

Likewise, we can define notions of best and worst because, in the ordering, we find an upper and lower bound. Since we have a finite ordering of classes of equivalent actions, we know that if $\instr{w}{\alpha}{\phi}\neq\emptyset$, there are $\Delta,\Gamma\in \instr{w}{\alpha}{\phi}$ such that $\Delta^{\alpha}$ is an upper bound, and $\Gamma^{\alpha}$ is a lower bound of $\better{w}{\phi}$. In other words, there are no $\Theta,\Sigma\in\instr{w}{\alpha}{\phi}$ such that $\Theta^{\alpha} \sbetter{w}{\alpha} \Delta^{\alpha}$, respectively $\Gamma^{\alpha}\sbetter{w}{\alpha}\Sigma^{\alpha}$. Note that it can be that $\success{w}{\alpha}{\phi}{\Delta}=\success{w}{\alpha}{\phi}{\Gamma}$ or that $\Delta=\Gamma$. Naively, we may say that an instrument $\Delta\in\instr{w}{\alpha}{\phi}$ in an upper bound of $\better{w}{\phi}$ is \textit{among the best} instruments for $\alpha$ at $w$, whereas an instrument $\Gamma\in\instr{w}{\alpha}{\phi}$ in a lower bound of $\better{w}{\phi}$ is \textit{among the worst} %available 
 instruments for $\alpha$ at $w$ for securing $\phi$.

There is an obvious objection to this naive approach: Suppose there are only two available $\phi$-instruments $\Delta,\Gamma\in\instr{w}{\alpha}{\phi}$ for $\alpha$ at $w$ such that $\success{w}{\alpha}{\phi}{\Delta}=1$ and $\success{w}{\alpha}{\phi}{\Gamma}=0.999$. Our naive definition tells us: (i) $\Delta$ is better than $\Gamma$ (for $\alpha$, in securing $\phi$ at $w$), (ii) $\Delta$ is among the best available $\phi$-instruments, and (iii) $\Gamma$ is among the worst available $\phi$-instruments. That $\Gamma$ is the `worst' instrument is arguably an overstatement. In fact, the argument can even be made that $\Gamma$ is an exceptionally better $\phi$-instrument than $\Delta$. For example, suppose that $\Delta$ was performed only once in the past by $\alpha$, securing $\phi$ and that $\Gamma$ was performed 1000 times, securing $\phi$ 999 times. Then, one should be able to conclude that $\Gamma$ is (by far) the best instrument. In particular, $\Gamma$ has proven to be more \textit{reliable} in producing the desired outcome. This observation motivates the instalment of \textit{thresholds}, which filter out insufficient experience, ensuring a certain quality standard of the instruments evaluated.

\medskip
\noindent \textbf{Thresholds for best, worst, and better.} We consider two types of thresholds. First, we can impose a threshold $n$ that states the minimum amount of past witnesses of an instrument's application. Second, we can impose a threshold on the minimum success ratio of potential instruments. We begin by considering the first threshold and adopt the second approach in \sect~\ref{subsec:good}. 

Let $n$ refer to the threshold that needs to be met by the total amount of witnesses $|\achiev{w}{\alpha}{\phi}{\Delta}\cup \fail{w}{\alpha}{\phi}{\Delta}|$. We write $\instr{w}{\alpha}{\phi,n}$ to denote the set of $\phi$-instruments satisfying threshold $n$, as defined in (\ref{defq:instr_n}).

{\setstretch{\s}
\begin{enumerate}
 [label={},ref=$\mathsf{d\arabic*}$, leftmargin=0.5cm]
   \setcounter{enumi}{23}
   \item\label{defq:instr_n}
   \textit{Available $\phi$-instruments for $\alpha\in\agents$ at $w$ (with threshold $n$)} \hfill(\ref{defq:instr_n})\\
   $\instr{w}{\alpha}{\phi,n}:=\{\Delta\  | \
   \Delta\in \instr{w}{\alpha}{\phi} 
 $ and $
|\achiev{w}{\alpha}{\phi}{\Delta}\cup \fail{w}{\alpha}{\phi}{\Delta}|\geq n\}$
   \end{enumerate}}
   
\noindent Based on (\ref{defq:instr_n}), we can refine the notion of `betterness' as follows:

{\setstretch{\s}
\begin{enumerate}
 [label={},ref=$\mathsf{d\arabic*}$, leftmargin=0.5cm]
   \setcounter{enumi}{24}
   \item\label{defq:better_n}
   \textit{Better $\phi$-instruments, relative to threshold $n\in\mathbb{N}$}\hfill(\ref{defq:better_n})\\
\textit{For each} $\Delta,\Gamma\in\instr{w}{\alpha}{\phi,n}$, $\Delta^{\alpha}\better{w}{\phi,n}\Gamma^{\alpha}$ \ifandonlyif $\success{w}{\alpha}{\phi}{\Delta} \geq \success{w}{\alpha}{\phi}{\Gamma}$
\end{enumerate}}

We interpret $\Delta^{\alpha}\better{w}{\phi,n}\Gamma^{\alpha}$ in (\ref{defq:better_n}) as `$\Delta^{\alpha}$ is a \textit{weakly better} $\phi$-instrument than $\Gamma^{\alpha}$ at $w$ given threshold $n$'. Imposing a threshold installs a quality control in providing axiological judgments of instrumentality. The undesirable consequences of the $0.999$ success rate example in the previous section can now be excluded by stipulating a threshold of any value $n> 1$, %excluding the single application case with success rate $1$ and 
identifying $\Gamma$ as the best $\phi$-instrument for $\alpha$. 
In (\ref{defq:best_n}) and (\ref{defq:worst_n}) below, we define the sets $\best{w}{\alpha}{\phi}{n}$ for `among the best' and $\worst{w}{\alpha}{\phi}{n}$ for `among the worst', which contain those instruments in the upper, respectively lower bound of the ordering, satisfying the imposed threshold.

{\setstretch{\s}
\begin{enumerate}
 [label={},ref=$\mathsf{d\arabic*}$, leftmargin=0.5cm]
   \setcounter{enumi}{25}
   \item\label{defq:best_n}
   \textit{Best $\phi$-instruments, relative to threshold $n\in\mathbb{N}$} \hfill(\ref{defq:best_n})\\
$\best{w}{\alpha}{\phi}{n} := \{ \Delta\in \instr{w}{\alpha}{\phi,n} \ | \ $ for each $\Gamma\in \instr{w}{\alpha}{\phi,n},  \Delta^{\alpha} \better{w,n}{\phi}\Gamma^{\alpha}\}$
\item\label{defq:worst_n}
\textit{Worst $\phi$-instruments, relative to threshold $n\in\mathbb{N}$} \hfill(\ref{defq:worst_n})\\
$\worst{w}{\alpha}{\phi}{n}\ \! {=}\ \! \{ \Delta\in \instr{w}{\alpha}{\phi,n} \ | $ for each $\Gamma\in \instr{w}{\alpha}{\phi,n},  \Gamma^{\alpha} \better{w,n}{\phi}\Delta^{\alpha}\}$
\end{enumerate}}

Reconsider the example in \fig~\ref{fig:success_ratio}. Depending on the threshold applied, either $\Delta$ or $\Gamma$ will qualify as among the best (worst) $\phi$-instruments at $w_6$. For instance, a threshold of $n=3$ excludes $\Gamma$ as a potential `best' $\phi$-instrument, i.e., $\Gamma\in\instr{w_6}{\alpha}{\phi,2}$, but $\Gamma\not\in\instr{w_6}{\alpha}{\phi,3}$. In fact, we find that $\Gamma\in\best{w_6}{\alpha}{\phi}{2}$ and $\Delta\in\worst{w_6}{\alpha}{\phi}{2}$, whereas $\Delta\in\best{w_6}{\alpha}{\phi}{3}$. Furthermore, observe that $\Delta\not\in\worst{w_6}{\alpha}{\phi}{3}$ since for $\Theta=\Delta\cup\Gamma$ we have $\Theta\in \instr{w_6}{\alpha}{\phi,3}$ and $\success{w_6}{\alpha}{\phi}{\Delta} > \success{w_6}{\alpha}{\phi}{\Theta}$.

\subsubsection{Thresholds and Good Instruments}\label{subsec:good}

In this section, we suggest possible ways to formally address the assessment of \textit{good}-instruments (item (4) of \dfn~\ref{def:philosophical_instruments}). As discussed in \sect~\ref{Sec:philosophy_agency}, for von Wright, comparative judgments are \textit{objective} since they depend on empirical data (e.g., collecting past experiences) and logical orderings (e.g., success ratio orderings). However, such judgments become more problematic when we address the label `good'. % (see \cite[Ch.~1]{Wri72}). 

Von Wright determines good instruments on the basis of their `good-making properties'. Recall, how \textit{well} a knife cuts depends on the good-making property associated with `cutting': the sharpness of the knife. Such value judgments concerning knives, von Wright argues, can be objectively assessed, especially when we order an available set of knives according to their sharpness. Without such a comparative set, things become more difficult since we need to define a minimal degree of `sharpness' that enables `cutting well'. Von Wright addresses this degree through \textit{causal properties} \cite[p.~26]{Wri72}, which determine the causal relation between `sharpness' and `cutting'---a relationship that may be objectively assessed through empirical investigations. Then, a knife cuts `well' or qualifies as a \textit{good} cutting-instrument, whenever it has the causal property of sharpness needed for cutting (note that this property depends on what needs to be cut). 

We mention two problems related to the above approach. First, such judgments of \textit{goodness} are a special case of comparative judgments. That is, the goodness reflected in a causal property is a goodness relative to (compared to) a threshold, for instance, the minimum amount of sharpness for cutting vegetables. %in order to relate causally to sharpness. 
In other words, judgments of instrumental goodness remain comparative but objective with respect to an external criterion, i.e., a threshold. Therefore, the term `good' employed here does not differ from any other use of `good' which, for instance, is defined relative to some theory of ethics. Still, the fact that such a threshold (the causal property) is rooted in experience allows for the judgment to be objectively evaluated. This observation is compatible with von Wright's ideas due to his emphasis on the logical nature of such judgments. 

Second, for an agent in a time-limited decision-making situation, determining the causal properties of an instrument (such as sharpness) may be an overburdening task. % for the agent's cognitive capacities. 
An agent that aims at cutting, say, a piece of paper may not have access to determining the causal properties of a different knife's sharpness relative to the robustness of the piece of paper. Although causal properties form objective criteria for obtaining accurate judgments of instrumentality from a theoretical point of view, from the agent's point of view, such criteria are unrealistic and impractical. We find this to be a strong reason for using the agent's \textit{personal experience} with an instrument (such as the knives available for cutting) as a criterion for judging instrumentality.

The question that remains is: What qualifies as sufficient experience for judging an instrument as \textit{good} for a given purpose?
We provide two notions of threshold that serve to denote sufficient experience: first, we propose a threshold on the success ratio of a given instrument and, second,  we combine the former with the notion of a minimum amount of past witnesses (see % developed in 
\sect~\ref{subsect:better}).

{\setstretch{\s}
\begin{enumerate}
 [label={},ref=$\mathsf{d\arabic*}$, leftmargin=0.5cm]
   \setcounter{enumi}{27}
\item\label{defq:good_instruments_ratio}\textit{Good$_n$ $\phi$-instruments, for $\alpha\in\agents$ at $w$} (with $n\in\mathbb{R}$ and $0\leq n\leq 1$)\\
$\good{w}{\alpha}{\phi,n} := \{ \Delta \ | \ \Delta\in \instr{w}{\alpha}{\phi}$ and $\success{w}{\alpha}{\phi}{\Delta} \geq n \}$ 
\hfill(\ref{defq:good_instruments_ratio})
\end{enumerate}
}

Definition (\ref{defq:good_instruments_ratio}) tells us that an action $\Delta$ is a \textit{good}$_n$ $\phi$-instrument for $\alpha$ at $w$, whenever `$\Delta$ qualifies as a $\phi$-instrument for $\alpha$ at $w$, and its success ratio satisfies threshold $n$'. Observe that the initial definition of (excellent) instruments  (\ref{defq:f-instr}) and (\ref{defq:e-f-instr})---i.e., items (1) and (2) of \dfn~\ref{def:philosophical_instruments}---are limiting cases of \textit{good}$_n$ instruments, namely, where the success ratios are $n>0$ and $n=1$, respectively. The relation between `good', (\ref{defq:f-instr}) and (\ref{defq:e-f-instr}) is expressed through the following:

{\setstretch{\s}
\begin{enumerate}
 % [label={},ref=$\mathsf{d\arabic*}$, leftmargin=0.5cm]
 %   \setcounter{enumi}{28}
\item[] %\label{defq:relations_good1_best_excellent}
$\Delta\in\good{w}{\alpha}{\phi,1}$ \ifandonlyif $w\models [\Delta^{\alpha}]^{ex-instr}_{\infty}\phi$ \ifandonlyif $\Delta\in\best{w}{\alpha}{\phi}{1}$ %\hfill(\ref{defq:relations_good1_best_excellent})

\item[]  %\label{defq:relations_good1_basic_instr}
$\Delta\in\good{w}{\alpha}{\phi,n}$ \ifandonlyif $w\models [\Delta^{\alpha}]^{instr}_{\infty}\phi$ (with $0<n\leq 1$)% \hfill(\ref{defq:relations_good1_basic_instr})
\end{enumerate}}

A \textit{poor} instrument can be defined in two ways: it is either an instrument failing to meet a `poorness'-threshold $n$ or an instrument that fails to qualify as good. We opt for the first approach---represented by (\ref{defq:poor_instr1})---to leave room for instruments that are considered neither good nor poor. 

{\setstretch{\s}
\begin{enumerate}
 [label={},ref=$\mathsf{d\arabic*}$, leftmargin=0.5cm]
   \setcounter{enumi}{28}
   \item\label{defq:poor_instr1}
\textit{Poor$_n$ $\phi$-instruments, for $\alpha\in\agents$ at $w$} ($n\in\mathbb{R}$ and $0\leq n\leq 1$)\\
$\poor{w}{\alpha}{\phi,n} := \{ \Delta \ | \ \Delta\in \instr{w}{\alpha}{\phi}$ and $\success{w}{\alpha}{\phi}{\Delta} \leq n \}$\hfill(\ref{defq:poor_instr1})
\end{enumerate}}

We face the same objection encountered while discussing naive betterness in \sect~\ref{subsect:better}. Here too, an instrument that proved itself a certain number of times could be considered more \textit{reliable} and thus more eligible for the label `good'. Therefore, in some cases, we may need to expand definition (\ref{defq:good_instruments_ratio}) with a second threshold imposing a minimum amount of experience with the instrument in question. The resulting definition is presented below (\ref{defq:good_instruments_ratio+}). %We use pairs $(n,m)$ to refer to double thresholds. 
We read $\Delta\in\good{w}{\alpha}{\phi,n,m}$ as `at $w$ for $\alpha$, action $\Delta$ is a \textit{good} $\phi$-instrument having a success ratio meeting $n$ and a minimum amount of past witnesses $m$'. 

{\setstretch{\s}
\begin{enumerate}
 [label={},ref=$\mathsf{d\arabic*}$, leftmargin=0.5cm]
   \setcounter{enumi}{29}
   \item\label{defq:good_instruments_ratio+}\textit{Good$^m_n$ $\phi$-instruments, for $\alpha\in\agents$ at $w$} ($n\in\mathbb{R}$, $0\leq n\leq 1$, $m\in\mathbb{N}$)\\
$\good{w}{\alpha}{\phi,n,m} := \{ \Delta \ | \ \Delta\in \instr{w}{\alpha}{\phi,m}$ and $\success{w}{\alpha}{\phi}{\Delta} \geq n \}$\hfill(\ref{defq:good_instruments_ratio+})
\end{enumerate}}

\noindent The modified definition of `poor' instruments is similarly obtained from (\ref{defq:poor_instr1}).

Last, as an illustration of the above machinery, consider the $\lae$-model provided in \fig~\ref{fig:success_ratio}. Suppose agent $\alpha$ desires to secure $\phi$ at $w_6$. Then, suppose the minimal required success ratio is $n=0.75$. In that case, at $w_6$ we find that only $\Gamma\in\good{w_6}{\alpha}{\phi,n}$ since $\success{w_6}{\alpha}{\phi}{\Gamma}=1\geq 0.75$. Action $\Delta$ fails to qualify due to $\success{w_6}{\alpha}{\phi}{\Delta}=\frac{2}{3} < 0.75$. 
%Depending on whether (\ref{defq:poor_instr2}) is adopted as a notion of `poorness' we would even have $\Delta\in\poor{w_6}{\alpha}{\phi}$. 
Suppose we impose the additional constraint that the minimum amount of witnesses is $3=m$. In that case, both $\Gamma,\Delta\not\in\good{w_6}{\alpha}{\phi,n,m}$ since $\Gamma$ does not satisfy $m$ and $\Delta$ doesn't satisfy $n$. Still, despite $\Delta$ not being able to qualify as a `good' $\phi$-instrument, it is nevertheless the `best' $\phi$-instrument for $\alpha$ at $w_6$, i.e., $\Delta\in\best{w_6}{\alpha}{\phi}{m}$.

In conclusion, whether an instrument is considered `good' or `poor' depends on its evaluative criteria. Nevertheless, via adopting different, yet combinatory, thresholds (i)-(iii), we can draw meaningful conclusions concerning the (comparative) value of the instruments at an agent's disposal. 
\begin{itemize}
    \item[(i)] the depth of past experience, e.g., (\ref{defq:c-instr}); 
    \item[(ii)] the minimum amount of witnesses, e.g., (\ref{defq:better}); 
    \item[(iii)] the minimum rate of success, e.g., (\ref{defq:good_instruments_ratio}).
\end{itemize}
Furthermore, we also saw that a `best' $\phi$-instrument $\Delta$ is not necessarily a good$_n$ $\phi$-instrument if the success ratio expressed by threshold $n$ is not met by $\Delta$. Alternatively, one could consider defining `best' in terms of ordering only those instruments that meet the threshold for qualifying as a good instrument. We leave such considerations for future work.

\subsection{Defeasibility of Instrumentality Judgments}\label{defeasibility}

Instrumentality, as defined in this work, is a defeasible notion in three ways. First, depending on the length of the interval considered to evaluate the past, an instrument $\Delta$ may fail to qualify as an (excellent) $\phi$-instrument once the interval is shortened. Furthermore, it can be easily checked that the excellence of instruments may also fail to be preserved when the interval is extended.  

Second, concerning the future, an instrument may fail to remain classified as an (excellent) instrument, either because an agent changes her expectations or, in the case of excellent candidate instruments, because the instrument has failed to produce the desired end in the meantime. (We note that this holds independent of whether time continues, i.e., whether $\lozenge\top$ holds.)  Nevertheless, plain candidate instrumentality is preserved over time when experience is accumulated. That is, once an instrument proves to be a candidate instrument, it remains a candidate instrument. %(for as long as `time continues'). 

The third and foremost defeasible aspect of our formalised instrumentality relations arises through the use of expectations. Namely, 
%following von Wright,
 the conjectural element of instrumentality judgments imposes a defeasible characteristic on such judgments: although universal statements can be objectively true for the past, such statements remain uncertain with respect to the future. %\footnote{Defeasibility of judgments concerning the future is discussed by von Wright in \cite{Wri57}.}
 The defeasibility of expectations consists in the possibility that, although an agent $\alpha$ expects that the (excellent) instrument will serve its intended end once again at the moment of evaluation $w$, the actual %\matold{overall foreseen} 
successor of $w$ is such that that instrument fails to deliver the purpose. These cases reveal a discrepancy between $\alpha$'s expectations and the actual future. Defeasibility with respect to the actual future 
is demonstrated by the fact that the formula %(\ref{defq:defeasibility3})
below is not a  
theorem of $\lae$.

{\setstretch{\s}
\begin{enumerate}
 % [label={},ref=$\mathsf{d\arabic*}$, leftmargin=0.5cm]
 %   \setcounter{enumi}{18}
   \item[]
   %\label{defq:defeasibility3}
   $[\Delta^{\alpha}]_{n}^{ex-instr}\phi\rightarrow \nbox (t(\Delta^{\alpha})\rightarrow \phi)$%\hfill(\ref{defq:defeasibility3})
   \end{enumerate}}

We close this section with a clarification concerning expectations: the conjectural element refers to the agent's expectations \textit{at the %present 
moment of evaluation, which is what an agent expects with respect to the immediate future of that moment.}  
By contrast, in %In 
evaluating the past, we must ignore the agent's past expectations in selecting the agent's relevant experience. 
In fact, a series of unexpected events in the past %, which 
may have led the agent to the conviction that a particular instrument is suitable for a specific purpose. That is, the agent may learn about instruments through unexpected events. For instance, suppose that the moment of evaluation $w$ is an immediate successor of a moment $w'$: then, it might be the case that at $w'$, the agent did not expect $w$ to obtain. Regardless, once $w$ obtains, the agent gains new experience concerning instrumentality judgments, as she always does whenever time passes.

\section{Closing Remarks}\label{sec:closing_remarks} 
We point out that we did not define von Wright's notion of \textit{bad} instruments \cite[p.~23, p.~35]{Wri72}. An available instrument (possibly a good or excellent instrument) is qualified as bad whenever a (side-)effect of that instrument would be undesirable. For example,  negotiations may be a good instrument for ending a war. Using an atomic bomb may even be excellent as an instrument. Still, the latter is a bad instrument because it will additionally lead to the destruction of the planet and the death of many (see \cite{Wri72} for a discussion). Here, the label `bad' is assigned to the instrument based on its (additional) consequences. These may, for instance, be certain moral or social values that are violated.
 
Moreover, there is the potential to provide \textit{deontic extensions} of the logic $\lae$ (the Greek d\'eon refers to `what is binding', i.e., duty). For instance, deontic concepts such as `obligation' and `prohibition' can be incorporated in $\lae$ through violation constants, e.g., an action (outcome) is obligatory if and only if that action's (outcome's) complement entails a violation. This approach is known as Anderson's reduction of deontic logic \cite{AndMoo57}. Such constants denote that the agent is in a violation state. 
A deontic extension of the basic $\laei$ logic from \cite{BerPas18} was developed in \cite{BerLyoOli20}. In that system, various instrumentality notions concerning obligations and prohibitions were introduced. For instance, in addition to the traditional distinction between `ought to be' and `ought to do'---respectively, obligations about states of affairs and obligations about actions---the involvement of instrumentality statements allows for a third category called `norms of instrumentality'. These are norms that oblige or prohibit a particular action as a means to achieving a particular end. To give an example, the norm `It is prohibited to use nonpublic information as an instrument to acquire financial profit on the stock market' (known as the law on ‘insider trading’) forbids the use of such information only as a \textit{means} to attain financial profit. The only temporal operator employed in \cite{BerLyoOli20} is the immediate successor modality. 

Combining the above two observations, one can extend the present work to a deontic setting which allows for reasoning with prohibitions that forbid those instruments that have deontically `bad' consequences (e.g. violations or sanctions) despite being `good' instruments. Another direction would be to define obligatory actions in terms of those `good' instruments (or best) that will secure an obligatory outcome. Here one can think of acquiring different notions depending on whether the agent's expectations will be involved.

\bibliographystyle{splncs04}

\end{document}